\documentclass[10pt]{amsart}
\usepackage{amssymb}
\usepackage{dsfont}
\usepackage{amscd}
\usepackage[mathscr]{euscript} % for the power-set P
\usepackage[all]{xy}
\usepackage{url}
\usepackage{stmaryrd}
\usepackage{comment}
\usepackage[retainorgcmds]{IEEEtrantools}
\usepackage{hyperref}
\usepackage{graphicx}
\title{Existentially closed fields with ${G}$-derivations}
\author[D. HOFFMANN]{Daniel Hoffmann$^{\dagger}$}
\thanks{2010 \textit{Mathematics Subject Classification}. Primary 13N15, 03C60; Secondary 14L15.}
\thanks{\textit{Key words and phrases}. Hasse-Schmidt derivations, group scheme actions, existentially closed structures.}
\address{$^{\dagger}$Instytut Matematyczny\\
Uniwersytet Wroc{\l}awski\\
Wroc{\l}aw\\
Poland}
\email{daniel.hoffmann@math.uni.wroc.pl}
\author[P. KOWALSKI]{Piotr Kowalski$^{\spadesuit}$}
\thanks{$^{\spadesuit}$Supported by T\"{u}bitak grant 2221 and NCN grant 2012/07/B/ST1/03513. Partially supported by ANR Modig (ANR-09-BLAN-0047)}
\address{$^{\spadesuit}$Instytut Matematyczny\\
Uniwersytet Wroc{\l}awski\\
Wroc{\l}aw\\
Poland}
\email{pkowa@math.uni.wroc.pl} \urladdr{http://www.math.uni.wroc.pl/\textasciitilde pkowa/ }

\DeclareMathOperator{\locus}{locus}
  
 \DeclareMathOperator{\aut}{Aut} \DeclareMathOperator{\id}{id}
 \DeclareMathOperator{\fr}{Fr} \DeclareMathOperator{\lie}{Lie}

\DeclareMathOperator{\ch}{char}  
 
 \DeclareMathOperator{\alg}{alg}

\DeclareMathOperator{\coli}{\underrightarrow{\lim}}

\DeclareMathOperator{\ev}{ev}

\DeclareMathOperator{\spec}{Spec}
\DeclareMathOperator{\sep}{sep}

\DeclareMathOperator{\sch}{SCH}
\DeclareMathOperator{\ddf}{DF}\DeclareMathOperator{\dcf}{DCF}\DeclareMathOperator{\scf}{SCF}
\newtheorem{theorem}{Theorem}[section]
\newtheorem{prop}[theorem]{Proposition}
\newtheorem{lemma}[theorem]{Lemma}
\newtheorem{cor}[theorem]{Corollary}
\newtheorem{fact}[theorem]{Fact}
\theoremstyle{definition}
\newtheorem{definition}[theorem]{Definition}
\newtheorem{example}[theorem]{Example}
\newtheorem{remark}[theorem]{Remark}

\newtheorem{notation}[theorem]{Notation}

\begin{document}

\newcommand{\twoc}[3]{ {#1} \choose {{#2}|{#3}}}
\newcommand{\thrc}[4]{ {#1} \choose {{#2}|{#3}|{#4}}}
\newcommand{\Zz}{{\mathds{Z}}}
\newcommand{\Ff}{{\mathds{F}}}
\newcommand{\Cc}{{\mathds{C}}}
\newcommand{\Rr}{{\mathds{R}}}
\newcommand{\Nn}{{\mathds{N}}}
\newcommand{\Qq}{{\mathds{Q}}}
\newcommand{\Kk}{{\mathds{K}}}
\newcommand{\Pp}{{\mathds{P}}}
\newcommand{\ddd}{\mathrm{d}}
\newcommand{\Aa}{\mathds{A}}
\newcommand{\dlog}{\mathrm{ld}}
\newcommand{\ga}{\mathbb{G}_{\rm{a}}}
\newcommand{\gm}{\mathbb{G}_{\rm{m}}}
\newcommand{\gaf}{\widehat{\mathbb{G}}_{\rm{a}}}
\newcommand{\gmf}{\widehat{\mathbb{G}}_{\rm{m}}}
\newcommand{\gdf}{\mathfrak{g}-\ddf}
\newcommand{\gdcf}{\mathfrak{g}-\dcf}
\newcommand{\fdf}{F-\ddf}
\newcommand{\fdcf}{F-\dcf}

\maketitle
\begin{abstract}
We prove that the theories of fields with Hasse-Schmidt derivations corresponding to actions of formal groups admit model companions. We also give geometric axiomatizations of these model companions.
\end{abstract}
\section{Introduction}
\noindent
In this paper, we describe the model theory of fields with Hasse-Schmidt derivations (abbreviated as \emph{HS-derivations} in the sequel) obeying iterativity conditions coming from the actions of formal groups. We consider ``$e$-dimensional HS-derivations in a generalized sense'' (see e.g. \cite[Def. 2.12]{HrFro}). This approach includes the case of $e$-tuples of the usual HS-derivations. (Actually, both approaches are equivalent in the iterative case, see Remark \ref{onedimenough}.)
%{\bf equivalent to, for $G$-iterative ones???}. We use ideas and some techniques from \cite{HK1} ({\bf as well as \cite{Zieg2}??}).
\\
\\
One could wonder why do the iterativity conditions help to understand the first-order theories of the fields with HS-derivations. The main reason is that the iterativity conditions enable us to characterize the \'{e}tale extensions of fields in a first order fashion, see Lemma \ref{constantdisjoint} (originating from \cite[Cor. 2.2]{Zieg2}). Such a characterization is crucial for the quantifier elimination results in Section \ref{secmc}.
%{\bf Iterativity generalizes and justifies (!!) the ``$\partial^{(p)}=0$'' condition!!!}
\\
\\
We consider both truncated and full HS-derivations. The iterativity rules considered in this paper are governed by (finite in the truncated case) formal groups. We describe the model companions of the theories of fields with such HS-derivations mostly using the ideas from \cite{Zieg2}. Then, we extend the results about the geometric axiomatizations  \cite{K3} from the case  of ``additive'' iterativity ($F=\widehat{\ga^e}$) to the case of an arbitrary $F$-iterative rule ($F$ is a formal group). We address the question whether the theories we obtain are bi-interpretable with the theory of separably closed fields with a fixed imperfection invariant (it is the case in \cite{Zieg2}). It turns out that we get such a bi-interpretability result for formal groups which are the formalizations of algebraic groups.
It is not clear for us what happens for the other formal groups. We also discuss the notion of ``canonical $G$-tuples'' (see Definition \ref{defpbasis}), which generalizes Ziegler's notion of canonical $p$-basis \cite{Zieg2},  and possible generalizations of our methods to the context of \cite{MS}.
\\
\\
The paper is organized as follows. In Section \ref{secdef}, we introduce the notation and conventions which will be used in the paper. In Section \ref{secsetup}, we develop an algebraic theory of (iterative, truncated, multi-dimensional) HS-derivations. In Section \ref{secmc}, we apply the results from Section \ref{secsetup} to obtain a description of the model complete theories of fields with different types of HS-derivations. In Section \ref{secaxioms}, we give geometric axioms for the theories considered in Section \ref{secmc}. In Section \ref{msecop}, we speculate about possible extensions of the results of this paper to more general contexts.

\subsection{Formal group actions and model theory}
For convenience of a reader not familiar with commutative algebra around the theory of formal groups, we provide in this section an elementary argument (originating from an observation of Matsumura \cite[Section 27]{mat}) explaining how formal group scheme actions can be understood as HS-derivations satisfying an iterativity rule.

For simplicity, we only consider the one-dimensional case.
A sequence $(D_i:R\to R)_{i\in \Nn}$ is an HS-derivation (over $k$) if and only if the corresponding map
$$\mathbb{D}:R\to R\llbracket X\rrbracket,\ \ \ \mathbb{D}(r)=\sum_{i=0}^{\infty}D_i(r)X^i$$
is a ring ($k$-algebra) homomorphism and a section of the projection map  $R\llbracket X\rrbracket\to R$. The standard iterativity condition may be expressed using the following diagram
\begin{equation*}
\xymatrix{R \ar[rr]^{\mathbb{D}} \ar[d]_{\mathbb{D}} & & R\llbracket X\rrbracket \ar[d]^{\mathbb{D}\llbracket X\rrbracket}
 \\ R\llbracket X\rrbracket\ar[rr]^{X\mapsto X+Y} & & R\llbracket X,Y\rrbracket,}
\end{equation*}
i.e. $\mathbb{D}$ is iterative if and only if the diagram above is commutative. Clearly, the additive \emph{formal group law} $X+Y$ is crucial for the standard iterativity rule above. (A formal group law over $k$, see e.g. \cite[Chapter IV.2]{Si}, is a power series $F$ over $k$ in two variables satisfying formally the group axioms, e.g. $F(F(X_1,X_2),X_3)=F(X_1,F(X_2,X_3))$.) The diagram above may be interpreted as an action of the formal additive group scheme on the scheme $\spec(R)$. We explain it below in an easier case of \emph{truncated} HS-derivations (see Section \ref{sectrun}), which correspond to actions (in the category of schemes) of finite group schemes.
\\
We start from the truncated iterativity diagram, which is just a truncation of the diagram above expressing the standard iterativity of HS-derivations
\begin{equation*}
\xymatrix{R \ar[rr]^{\mathbb{D}} \ar[d]_{\mathbb{D}} & & R[v_m] \ar[d]^{\mathbb{D}[v_m]}
 \\ R[w_m]\ar[rr]^{c_R} & & R[w_m,v_m].}
\end{equation*}
Here $c_R(w_m)=v_m+w_m$ gives a Hopf algebra structure over $R$. Using the tensor product over the base field $k$, we obtain the following diagram.
\begin{equation*}
\xymatrix{R \ar[rr]^{\mathbb{D}\ \ \ \ } \ar[d]_{\mathbb{D}} & & k[v_m]\otimes R \ar[d]^{\id_{k[v_m]}\otimes \mathbb{D}}
 \\ k[v_m]\otimes R\ar[rr]^{c\otimes \id_R\ \ \ \ } & & k[v_m]\otimes k[v_m]\otimes R}
\end{equation*}
Going to the opposite category (of $k$-schemes), we see that we get exactly the diagram expressing the mixed associativity of a group scheme action
\begin{equation*}
\xymatrix{R  & & \ar[ll]_{\widetilde{\mathbb{D}}}  \mathfrak{g}\times X
 \\ \mathfrak{g}\times X  \ar[u]^{\widetilde{\mathbb{D}}} & & \ar[ll]_{\mu \times \id_X} \ar[u]_{\id_{\mathfrak{g}}\times \widetilde{\mathbb{D}}} \mathfrak{g}\times \mathfrak{g}\times X,}
\end{equation*}
where $X=\spec(R)$ and $\mathfrak{g}=\ker(\fr^m_{\ga})$.

The considerations above lead to an interesting conclusion that actions of groups with ``no points'' (i.e. finite local group schemes as above), which seem to be very far from any model-theoretic considerations, are actually amenable to model-theoretic treatment; it is the main point of this paper.

%One could (and probably should) wonder how far this ...???... can go. A natural final level of generality in this context seems to be Kamensky's ...?.... %\cite{Kam}, which is apparently a dual setting to the one considered by Moosa and Scanlon in \cite{MS2} (more comments ... Section \ref{msecop}??). Such problems %will be picked up in a subsequent work.

\section{Definitions, notation and conventions}\label{secdef}
\noindent
In this section we introduce the notation and conventions which we are going to use throughout the paper. We also recall (or refer to) several standard notions.
\\
\\
In the entire paper, $k$ will be a perfect field of characteristic $p>0$ (unless we clearly say that $\ch(k)=0$). The category of \emph{affine group schemes} over $k$ is the category opposite to the category of \emph{Hopf algebras} over $k$ \cite[Section 1.4]{Water} (or it is the category of representable functors from $k$-algebras to groups, see \cite[Section 1.2]{Water}). A \emph{truncated group scheme} \cite{Chase2} over $k$ is an affine group scheme whose universe is  isomorphic to $\spec(k[X_1,\ldots,X_e]/(X_1^{p^m},\ldots,X_e^{p^m}))$.
%By \cite[14.4]{Water} such group schemes coincide with finite connected group schemes (called \emph{infinitesimal} in \cite{Demazure} and \emph{local} in %\cite{Manin}).
%\\
%{\bf 0. TRUE, OR THE LATTER IS MORE GENERAL??}
\begin{remark}\label{notrunzero}
If $\ch(k)=0$, then by a theorem of Cartier \cite[Section 11.4]{Water} all Hopf algebras over $k$ are reduced, so there are no truncated group schemes.
\end{remark}
\noindent
The category of \emph{formal groups} over $k$ is the category opposite to the category of \emph{complete Hopf algebras} over $k$ (or the category of representable functors from complete $k$-algebras to groups, see \cite[Chapter VII]{Hazew}). There is a correspondence between smooth formal groups (the underlying complete algebra is the power series algebra in $e$ variables) and formal group laws, where an $e$-dimensional formal group law over $k$, is a power series in $2e$ variables formally satisfying the group axioms, see \cite[Sect. 9.1]{Hazew}. Note that a truncated group scheme is both an affine group scheme and a formal group.
\\
\\
For the rest of the paper we fix the following.
\begin{itemize}
\item Let $m$ and $e$ be positive integers.

\item Let $\mathbf{X}$ denote the tuple of variables $(X_1,\ldots,X_e)$. For a tuple $\mathbf{n}=(n_1,\ldots,n_e)$ of natural numbers, we denote $X_1^{n_1}\ldots X_e^{n_e}$ by $\mathbf{X}^{\mathbf{n}}$.

\item Let $k[\mathbf{v}_m]$ denote the ring $k[\mathbf{X}]/(X_1^{p^m},\ldots,X_e^{p^m})$.

\item For a positive integer $l$, let $[l]$ denote the set $\{0,\ldots,l-1\}$.

\item Let $\mathfrak{g}$ be a group scheme over $k$ whose underlying scheme is $\spec(k[\mathbf{v}_m])$.

\item Let $R$ and $S$ be $k$-algebras.

\item Let $G$ be an algebraic group over $k$.

\item Let $V$ be a scheme over $k$.

\item Let $F$ be an $e$-dimensional formal group law over $k$.
\end{itemize}

\subsection{Truncations of group schemes}\label{gpschemetrun}
Let $\mathfrak{G}$ be an affine group scheme over $k$, $H$ the corresponding Hopf algebra and $\mathfrak{m}$ be the kernel of the counit map $H\to k$ (the \emph{augmentation ideal}). Using the base-change given by the automorphism $\fr^m:k\to k$ we get the affine group scheme $\mathfrak{G}^{\fr^m}$ over $k$ and a group scheme
morphism $\fr^m_{\mathfrak{G}}:\mathfrak{G}\to \mathfrak{G}^{\fr^m}$. Let $\mathfrak{G}[m]$ be the kernel of $\fr^m_{\mathfrak{G}}$ which is a truncated $k$-group scheme. (In the case of a commutative group scheme $A$, $A[m]$ often denotes the kernel of multiplication by $m$ and ``our'' $A[m]$ is often denoted by $A[\fr^m]$. Since we do not consider the kernel of multiplication by $m$ in this paper, we prefer our simplified notation.)

It corresponds to the quotient Hopf algebra $$H[m]:=H/\fr^m(\mathfrak{m})H.$$
We get a direct system of truncated $k$-group schemes $(\mathfrak{G}[n])_{n\in \Nn}$. If $\mathfrak{G}=G$, then $\coli(G[n])$ coincides with $\widehat{G}$, the formal group which is the formalization of $G$ (see \cite[Lemma 1.1]{Manin}).
\\
%Let $H$ be a Hopf algebra over $k$, the maximal ideal . Let us define
%$$H[m]:=H/\mathfrak{m}^{p^m}.$$
%Then $H[m]$ has a natural Hopf algebra structure (it is also a complete Hopf algebra) such that the quotient map $H\to H[m]$ is a Hopf algebra homomorphism [{\bf %reference}].
Similarly, for a complete Hopf algebra $\mathcal{H}$, we have the analogous quotient $\mathcal{H}[m]$ which is a Hopf algebra and also a complete Hopf algebra. Hence for a formal group $F$, we have a direct system of truncated group schemes $F[m]$ and in this case we get that $F=\coli F[m]$, see \cite[Lemma 1.1]{Manin} again.
\begin{remark}\label{intgs}
One may ask whether any truncated group scheme $\mathfrak{g}$ can be integrated i.e. whether there is a formal group law $F$ such that $F[m]=\mathfrak{g}$. For $e=1$, the answer is positive if and only if $\mathfrak{g}$ is commutative \cite[Corollary 5.7.4]{Hazew} (see \cite[Example 5.7.8]{Hazew} for an example of a non-commutative $\mathfrak{g}$).
\end{remark}
\begin{remark}\label{gm}
%\begin{enumerate}
%\item By \cite[Lemma 1.1]{Manin}, a formal group law $F$ gives a direct system $(F[m])_m$ of finite group schemes over $k$ (which may be also considered as %formal groups over $k$) and $F$ is the direct limit of $(F[m])_m$ in the category of formal groups over $k$.
%\item By \cite[Chapter I.2]{Manin}, any algebraic group $G$ over $k$, functorially gives a smooth formal group $\widehat{G}$, called \emph{the formalization of %$G$}. We denote $\widehat{G}[m]$ by $G[m]$.
For the truncated group scheme $\mathfrak{g}$, we get a finite direct system of truncated group schemes
$$0=\mathfrak{g}[0]<\mathfrak{g}[1]<\ldots <\mathfrak{g}[m-1]<\mathfrak{g}[m]=\mathfrak{g}.$$
The group schemes in this direct system may be described as follows. Let $i\in [m+1]$ and $\mathfrak{g}^{\fr^i}$ be the group scheme $\mathfrak{g}$ twisted
%\\
%{\bf 2. IS THIS ITEM NEEDED AT ALL? DEPENDS ON THE PROOF OF UNIQUE STRICT EXTENSION. NOW COVERED ABOVE IN GREATER GENERALITY!!!}
%\\
by the $i$-th power of the Frobenius map. Then we have a group scheme morphism $\fr_{\mathfrak{g}}^i:\mathfrak{g}\to \mathfrak{g}^{\fr^i}$ such that
$$\ker(\fr_{\mathfrak{g}}^i)=\mathfrak{g}[i],\ \ \ \fr_{\mathfrak{g}}^i(\mathfrak{g})=\mathfrak{g}^{\fr^i}[m-i].$$
%\end{enumerate}
\end{remark}

\section{Finite group schemes and iterative HS-derivations}\label{secsetup}
\noindent
In this section, we develop an algebraic theory of (iterative, truncated) multi-dimensional HS-derivations.

\subsection{Multi-dimensional truncated HS-derivations}\label{sectrun} We are going to use the following definition.
\begin{definition}\label{hstrdef}
\begin{enumerate}
\item  An \emph{$e$-dimensional HS-derivation on $R$ over $k$}  is a $k$-algebra homomorphism
$$\mathbb{D}:R\to R\llbracket \mathbf{X}\rrbracket$$
which is a section of the projection map
$$R\llbracket \mathbf{X}\rrbracket\to R,\ \ \ H\mapsto H(\mathbf{0}).$$

\item An \emph{$m$-truncated $e$-dimensional HS-derivation on $R$ over $k$}  is a $k$-algebra homomorphism
$$\mathbb{D}:R\to R[\mathbf{v}_m]$$
which is a section of the projection map $R[\mathbf{v}_m]\to R$.
\end{enumerate}
\end{definition}
\begin{remark}\label{edimrem}
\begin{enumerate}
\item From any $e$-dimensional HS-derivation $\mathbb{D}$ on $R$ over $k$ and any positive integer $n$, we get in an obvious way (i.e. by post-composing with the quotient map $R\llbracket \mathbf{X}\rrbracket\to  R[\mathbf{v}_n]$) an $n$-truncated $e$-dimensional HS-derivation on $R$ which we denote by $\mathbb{D}[n]$.

\item Let us denote $\mathbb{D}(r)$ by $\sum_{\mathbf{i}} D_{\mathbf{i}}(r)\mathbf{X}^{\mathbf{i}}$. Using such a notation, an $e$-dimensional HS-derivation on $R$ over $k$ is a sequence
$$\mathbb{D}=(D_{\mathbf{i}}:R\to R)_{\mathbf{i}\in \Nn^e}$$
satisfying the following properties:
\begin{itemize}
\item $D_{\mathbf{0}}=\id_R$,

\item each $D_{\mathbf{i}}$ is $k$-linear,

\item for any $x,y\in R$ we have
$$D_{\mathbf{i}}(xy)=\sum_{\mathbf{j}+\mathbf{k}=\mathbf{i}}D_{\mathbf{j}}(x)D_{\mathbf{k}}(y).$$
\end{itemize}

\item Each $e$-dimensional HS-derivation $\mathbb{D}$ on $R$ gives the following tuple of (1-dimensional) HS-derivations on $R$:
$$\mathbb{D}_1:=(D_{(i,0,\ldots,0)})_{i\in \Nn},\ \ldots \ ,\mathbb{D}_e:=(D_{(0,\ldots,0,i)})_{i\in \Nn}.$$
On the level of $k$-algebra maps, the above $m$-truncated HS-derivations correspond to the composition of $\mathbb{D}:R\to R\llbracket \mathbf{X}\rrbracket$ with the appropriate projection map $R\llbracket \mathbf{X}\rrbracket\to R\llbracket X\rrbracket$.

\item On the other hand, each $e$-tuple of (1-dimensional) HS-derivations on $R$ gives an $e$-dimensional HS-derivation on $R$, e.g. for $e=2$ and $\mathbb{D}_1,\mathbb{D}_2:R\to  R\llbracket X\rrbracket$ we get the $2$-dimensional HS-derivation on $R$ given by the composition below
    \begin{equation*}
 \xymatrixcolsep{2.5pc}\xymatrixrowsep{1.5pc}\xymatrix{R \ar[r]^{\mathbb{D}_1\ \ \ }  & R\llbracket X\rrbracket \ar[r]^{\mathbb{D}_2\llbracket X\rrbracket\ \ \ \  } &
R\llbracket X,Y\rrbracket .}
\end{equation*}
\noindent
However, not all $e$-dimensional HS-derivations on $R$ can be obtained in such a way. In fact an $e$-dimensional HS-derivation $\mathbb{D}$ is not necessarily determined by the $e$-tuple $\mathbb{D}_1,\ldots,\mathbb{D}_e$ from $(3)$. For example consider the (truncated) case when $p=2$ and $m=1$. If $\partial$ is a non-zero derivation on $R$, then the map
$$\mathbb{D}(r)=r+\partial(r)(X+(X^2))(Y+(Y^2))$$
is a non-zero $1$-truncated $2$-dimensional HS-derivation, but the corresponding $1$-truncated $1$-dimensional HS-derivations from $(3)$ are the zero maps.

\item All the above (for $n\leqslant m$) applies to $m$-truncated $e$-dimensional HS-derivations (after replacing ``$\Nn$'' with ``$[p^m]$'' and ``$\llbracket \mathbf{X}\rrbracket$'' with ``$[\mathbf{v}_m]$'').

\item We can extend Definition \ref{hstrdef} to define $m$-truncated $e$-dimensional HS-derivations from $R$ to $S$ (we do not require anything about the sections here).
    %We will need such a generality in the proof of Proposition \ref{strictetale}.

\item Our $m$-truncated $1$-dimensional HS-derivations correspond to the \emph{higher derivations of length $p^m-1$} from \cite{mat}.

\end{enumerate}
\end{remark}
\noindent
For the definition of an \'{e}tale map/algebra, the reader is advised to consult \cite[p. 193]{mat} (called ``0-\'{e}tale'' there). It is easy to see that the condition ``$N^2=0$'' from \cite[p. 193]{mat} may be replaced with the condition ``$N$ is nilpotent'' (see e.g. Remark on page 199 of \cite{mat0}).
\begin{prop}\label{etalehs}
 Assume that $R\to S$ is  an \'{e}tale $k$-algebra map. Then any ($m$-truncated) $e$-dimensional HS-derivation $\mathbb{D}$ on $R$ uniquely extends to an ($m$-truncated) $e$-dimensional HS-derivation $\mathbb{D}'$ on $S$.
\end{prop}
\begin{proof}
The proof goes almost exactly as in \cite[theorem 27.2]{mat}, so we will just sketch the main inductive step which makes clear how the \'{e}tale assumption is used. We apply the induction on the truncation degree. Assume that for $l\in \Nn$ (resp. $l<m$) we have extended $\mathbb{D}[l]$ to an $l$-truncated $e$-dimensional HS-derivation $\mathbb{D}':=(D'_{\mathbf{i}})_{\mathbf{i}<[p^l]^e}$ on $S$. Consider the following diagram
\begin{equation*}
 \xymatrixcolsep{2.5pc}\xymatrixrowsep{1.5pc}\xymatrix{& S[\mathbf{X}]/(X_1^{p^l},\ldots,X_e^{p^l}) &
\\ S  \ar[ur]^-{\mathbb{D}'} \ar@{-->}[rr]
 &  & S[\mathbf{X}]/(X_1^{p^{l+1}},\ldots,X_e^{p^{l+1}})  \ar[ul]_{\pi}
\\ & R \ar[ul]^{f} \ar[ur]_{\mathbb{D}[l+1]} &
}
\end{equation*}
where $\pi$ is the quotient map. Then $(\ker \pi)^{e+1}=0$, so $\ker(\pi)$ is nilpotent. Since the map $R\to S$ is \'{e}tale, we get a unique $k$-algebra map
$$S\to S[\mathbf{X}]/(X_1^{p^{l+1}},\ldots,X_e^{p^{l+1}})$$
completing the diagram above.
\end{proof}

\begin{remark}\label{defgen}
Proposition \ref{etalehs} enables us to generalize Definition \ref{hstrdef} in the following way.
\begin{enumerate}
\item Since the localization maps are \'{e}tale (see \cite[p. 193]{mat}), any ($m$-truncated) $e$-dimensional HS-derivation on $R$ uniquely extends to an ($m$-truncated) $e$-dimensional HS-derivation on a localization of $R$.

\item By $(1)$, we get a notion of an ($m$-truncated) $e$-dimensional HS-derivation on any scheme over $k$.

\item Proposition \ref{etalehs} generalizes to schemes over $k$.

\item The notion of an ($m$-truncated) $e$-dimensional HS-derivation on a scheme $V$ is a special case of the notion of a \emph{$\underline{\mathcal{D}}$-structure} on a scheme $V$, see \cite{MS}. We will discuss possible generalizations of the results of this paper to the context of \cite{MS} in Section \ref{msecop}.

\item It is easy to generalize the assumptions of Proposition \ref{etalehs} to include the case of ($m$-truncated) $e$-dimensional HS-derivations \emph{from $R$ to $S$}.
\end{enumerate}
\end{remark}
\begin{definition}\label{firstdefstrict}
If $\mathbb{D}$ is an ($m$-truncated) $e$-dimensional HS-derivation on $R$ over $k$, then we define the following.
\begin{enumerate}
\item The \emph{ring of constants} of $(R,\mathbb{D})$ is
$$\ker(D_{(1,0,\ldots,0)})\cap \ldots \cap \ker(D_{(0,\ldots,0,1)}).$$
Clearly, $R^p$ is contained in the ring of constants of $(R,\mathbb{D})$.

\item We call $(R,\mathbb{D})$ \emph{strict}, if the ring of constants of $\mathbb{D}$ coincides with $R^p$.

\item The \emph{ring of absolute constants} of $(R,\mathbb{D})$ is
$$\bigcap_{\mathbf{i}\neq \mathbf{0}}\ker(D_{\mathbf{i}}).$$
\end{enumerate}
\end{definition}
\begin{remark}
If $\mathbb{D}$ is an (resp. $m$-truncated) $e$-dimensional HS-derivation on $R$ over $k$, then $R^{p^{\infty}}$ (resp. $R^{p^m}$) is contained in the ring of absolute constants. It is easy to see (e.g. in the $m$-truncated case) considering $\mathbb{D}:R\to R[\mathbf{v}_m]$ as a ring homomorphism and taking the $p^m$-th power.
\end{remark}

\begin{notation}
The couple $(R,\mathbb{D})$ will be usually denoted by $\mathbf{R}$ and called an ($m$-truncated $e$-dimensional) \emph{HS-ring}. Similarly, we get the notions of \emph{HS-fields}, \emph{HS-extension}, etc.
\end{notation}

\subsection{Group scheme actions}
%New notation {\bf replace with the ``$\mathbf{v}_m$-notation''!!}
%$$R[\overline{X_1},\ldots,\overline{X_n},m]:= R[\overline{X_1},\ldots,\overline{X_n}]/I,$$
%where $I=(X_{1,1}^{p^m},\ldots,X_{n,e}^{p^m})$.
We introduce a notion which generalizes the notion of an $m$-truncated iterative HS-derivation from \cite{K3}.
\begin{definition}
\begin{enumerate}
\item A \emph{$\mathfrak{g}$-derivation on $V$} is a $k$-group scheme action of $\mathfrak{g}$ on $V$ (see Section 12 in \cite{MuAbel}).

\item A \emph{$\mathfrak{g}$-derivation on $R$} is a $\mathfrak{g}$-derivation on $\spec(R)$.

\item We naturally get the notions of a \emph{$\mathfrak{g}$-ring}, a \emph{$\mathfrak{g}$-field} and a \emph{$\mathfrak{g}$-extension}.
\end{enumerate}
\end{definition}
\begin{remark}\label{gderint}
A $\mathfrak{g}$-derivation on $R$ is the same as an $m$-truncated $e$-dimensional
HS-derivation on $R$ over $k$ satisfying a ``$\mathfrak{g}$-iterativity'' rule. It is easy to see that the trivial action of the unit morphism corresponds
to the condition $D_{\mathbf{0}}=\id_R$ and the diagram expressing the mixed associativity of the $k$-group scheme action $d$ is the following ``$\mathfrak{g}$-iterativity'' diagram
\begin{equation*}
\xymatrix{R \ar[r]^d \ar[d]_d & R[\mathbf{v}_m] \ar[d]^{d[\mathbf{v}_m]}
 \\ R[\mathbf{w}_m]\ar[r]^{c} & R[\mathbf{w}_m,\mathbf{v}_m]}
\end{equation*}
where $\mathbf{w}_m$ is another ``$m$-truncated $e$-tuple of variables'' and $c$ is the Hopf algebra comultiplication given by $\mathfrak{g}$. Therefore for an arbitrary $k$-scheme $V$, any $\mathfrak{g}$-derivation on $V$ is also an $m$-truncated $e$-dimensional HS-derivation on $V$ over $k$ in the sense of Remark \ref{defgen}(2).
\end{remark}

\begin{remark}\label{gderint2}
We will give another interpretation of the $\mathfrak{g}$-iterativity condition. Suppose that the matrix (in the standard basis) of the $k$-linear comultiplication map $c$ from Remark \ref{gderint} has the form $(c_{\mathbf{i},\mathbf{j}}^{\mathbf{k}})$. Suppose also that $\mathbb{D}=(D_{\mathbf{i}})_{\mathbf{i}}$ is an $m$-truncated $e$-dimensional HS-derivation on $R$ over $k$. Then $\mathbb{D}$ is a $\mathfrak{g}$-derivation if and only if for all $\mathbf{i},\mathbf{j}$ we have
$$D_{\mathbf{j}}\circ D_{\mathbf{i}}=\sum_{\mathbf{k}}c_{\mathbf{i},\mathbf{j}}^{\mathbf{k}}D_{\mathbf{k}}.$$
\end{remark}
\noindent
One easily shows the following.
\begin{fact}\label{frometo1}
If $\mathfrak{g}=\mathfrak{g}_1\times \ldots \times \mathfrak{g}_e$ (product of finite group schemes), then we have the following.
\begin{enumerate}
\item For any $\mathfrak{g}$-derivation $\mathbb{D}$ on $R$ and $i\leqslant e$, $\mathbb{D}_i$ from Remark \ref{edimrem} is a $\mathfrak{g}_i$-derivation and we have
$$D_{(i_1,\ldots,i_e)}=D_{1,i_1}\circ \ldots \circ D_{e,i_e}.$$
\item If for any any $i\leqslant e$, $\mathbb{D}_i$ is an $\mathfrak{g}_i$-derivation, and we have $D_{i,i'}\circ D_{j,j'}= D_{j,j'}\circ D_{i,i'}$ for all $i,j\leqslant e$ and $i',j'\in \Nn$, then the formula
$$D_{(i_1,\ldots,i_e)}=D_{1,i_1}\circ \ldots \circ D_{e,i_e}$$
defines a $\mathfrak{g}$-derivation.
\end{enumerate}
\end{fact}
\begin{example}\label{gitex}
We give below several examples of $\mathfrak{g}$-iterativity rules.
\begin{enumerate}
\item A sequence of $e$ commuting iterative $m$-truncated HS-derivations from \cite{K3} is the same as a $\ga^e[m]$-derivation (see Fact \ref{frometo1}).

\item Let $G$ be the unipotent algebraic group of dimension $2$ given by the cocycle
$$\frac{(X+Y)^p-X^p-Y^p}{p},$$
see \cite[p. 171]{serre1988algebraic}.
\\
Assume that $p=2$ and $\mathbb{D}$ is an $m$-truncated $2$-dimensional HS-derivation. Then $\mathbb{D}$ is a $G[m]$-derivation if and only if
$$D_{(k,l)}\circ D_{(i,j)}=\sum\limits_{t=0}^{\min(j,l)}\frac{(i+k+t)!}{i!k!t!}
\frac{(j+l-2t)!}{(j-t)!(l-t)!}\;\; D_{(i+k+t,j+l-2t)}.$$
In particular, we have the following formulas which actually describe the $G[m]$-iterativity rule fully:
$$D_{(0,j)}\circ D_{(i,0)}=D_{(i,j)}=D_{(i,0)}\circ D_{(0,j)},$$
$$D_{(k,0)}\circ D_{(i,0)}={i+k\choose i}D_{(i+k,0)},$$
$$D_{(0,1)}\circ D_{(0,i)}=(i+1)D_{(0,i+1)}+D_{(1,i-1)}.$$
The first author described in \cite{Hoff2} a modification of the theory of separably closed fields with higher derivations from \cite{MeWo} using the iterativity rules coming from algebraic groups similar to the one considered here (groups of Witt vectors).

\item Let $G$ be $\ga\rtimes \gm$, where the group operation on $\gm$ is given by $X+Y+XY$. Hence the group operation on $G$ is given by
$$(X_1,Y_1)*(X_2,Y_2)=(X_1+X_2+Y_1X_2,Y_1+Y_2+Y_1Y_2).$$
Let $\mathbb{D}$ be an $m$-truncated $2$-dimensional HS-derivation. Then $\mathbb{D}$ is a $G[m]$-derivation if and only if
$$D_{(k,l)}\circ D_{(i,j)} = \sum\limits_{t=0}^{\min(k,j)}\sum\limits_{s=0}^{\min(l,j-t)}
\frac{(i+k)!}{i!(k-t)!t!}\;\frac{(j+l-t-s)!}{(j-t-s)!(l-s)!s!}\;\; D_{(i+k,j+l-t-s)}.$$
In particular, we have
$$D_{(0,l)}\circ D_{(i,0)}=D_{(i,l)}.$$
But the above formula does not apply for the other choice of coordinates
$$D_{(1,0)}\circ D_{(0,1)}=D_{(1,1)}+D_{(1,0)}\neq D_{(1,1)}=D_{(0,1)}\circ D_{(1,0)}.$$
We also have the following ``additive coordinate'' rule
$$D_{(1,0)}\circ D_{(i,0)}=(i+1)D_{(1+i,0)}=D_{(i,0)}\circ D_{(1,0)},$$
and the ``multiplicative coordinate'' rule
$$D_{(0,1)}\circ D_{(0,i)}=(i+1)D_{(0,i+1)}+iD_{(0,i)}=D_{(0,i)}\circ D_{(0,1)}.$$
\end{enumerate}
\end{example}
\begin{remark}\label{onedimenough}
We see that all the $e$-dimensional HS-derivations in the example above are determined by the 1-dimensional HS-derivations $\mathbb{D}_1,\ldots,\mathbb{D}_e$ from Remark \ref{edimrem}(3). It may be shown (using Lemma \ref{newton} below) that this is the case for an arbitrary $F$-derivation (or a $\mathfrak{g}$-derivation).
\end{remark}\noindent
We will need more precise information about the ``structural constants'' from Remark \ref{gderint2}.
\begin{lemma}\label{newton}
Let $c_{\mathbf{i},\mathbf{j}}^{\mathbf{k}}$ be as in Remark \ref{gderint2}. For a tuple of natural numbers $\mathbf{n}=(n_1,\ldots,n_e)$, the sum $n_1+\ldots+n_e$ is denoted by $|\mathbf{n}|$. Then we have the following.
\begin{enumerate}
\item If $|\mathbf{k}|>|\mathbf{i}|+|\mathbf{j}|$, then $c_{\mathbf{i},\mathbf{j}}^{\mathbf{k}}=0$.

\item If $|\mathbf{k}|=|\mathbf{i}|+|\mathbf{j}|$ and $\mathbf{k}\neq \mathbf{i}+\mathbf{j}$, then $c_{\mathbf{i},\mathbf{j}}^{\mathbf{k}}=0$.

\item If $\mathbf{k}=\mathbf{i}+\mathbf{j}$, then
$$c_{\mathbf{i},\mathbf{j}}^{\mathbf{k}}={i_1+j_1 \choose i_1}\cdot \ldots \cdot {i_e+j_e \choose i_e}.$$
\end{enumerate}
\end{lemma}
\begin{proof}
It is clear for $\text{\textbf{i}}=\mathbf{0}$ or $\text{\textbf{j}}=\mathbf{0}$, so we assume that $\mathbf{i},\mathbf{j}\neq \mathbf{0}$. By a truncated version of the formula \cite[(14.1.1)]{Hazew}, we have (in the notation of Remark \ref{gderint})
$$c(\mathbf{w}_m)=\mathbf{w}_m+\mathbf{v}_m+\mathbf{s}_m,$$
where $\mathbf{s}_m=(S_1,\ldots,S_e)$ for some $S_1,\ldots,S_e$ belonging to the ideal $(\mathbf{w}_m\cdot \mathbf{v}_m)$. Therefore for every $r\in R$ we have
$$\sum_{\mathbf{i},\mathbf{j}} D_{\mathbf{j}}D_{\mathbf{i}}(r)\mathbf{v}_m^{\mathbf{i}}\mathbf{w}_m^{\mathbf{j}}=
\sum_{\mathbf{k}} D_{\mathbf{k}}(r)(\mathbf{w}_m+\mathbf{v}_m+\mathbf{s}_m)^{\mathbf{k}}.$$
We get the result by comparing the coefficients at $\mathbf{v}_m^{\mathbf{i}}\mathbf{w}_m^{\mathbf{j}}$.
\end{proof}
\begin{remark}
 Note that for every $\mathbf{i}, \mathbf{j}\neq \mathbf{0}$ we have the following:

 $$D_{\mathbf{j}}\circ D_{\mathbf{i}}={i_1+j_1 \choose i_1}\cdot\ldots\cdot{i_e+j_e \choose i_e}D_{\mathbf{i}+\mathbf{j}}
 +\mathcal{O}(D_{\mathbf{n}})_{0<|\mathbf{n}|<|\mathbf{i}+\mathbf{j}|},$$
 where $\mathcal{O}(D_{\mathbf{n}})_{0<|\mathbf{n}|<|\mathbf{i}+\mathbf{j}|}$ is a $k$-linear combination of
 $D_{\mathbf{n}}$ for $0<|\mathbf{n}|<|\mathbf{i}+\mathbf{j}|$. We consider the quantity $\mathcal{O}(\cdot)$ as a ``disturbance from the additive iterativity'', because for the additive iterativity condition this linear combination is always zero. Lemma 3.13 from \cite{HK1} regards the case of $e=1$.
\end{remark}
\begin{lemma}\label{absconst}
If $\mathbf{R}=(R,\mathbb{D})$ is a $\mathfrak{g}$-ring, then the ring of constants of $\mathbf{R}$ coincides with the ring of absolute constants of $(R,\mathbb{D}[1])$.
\end{lemma}
\begin{proof}
It follows from Lemma \ref{newton}, that for any $\mathbf{i}\neq \mathbf{0}$, $D_{\mathbf{i}}$ is a $k$-linear combination of the compositions of the derivations $D_{(1,0,\ldots,0)},\ldots,D_{(0,\ldots,0,1)}$ which gives the result.
\end{proof}
%\begin{lemma}\label{piterative}
%Let $\mathbb{D}$ be an $m$-truncated $e$-dimensional HS-derivation on $R$. If
%$$D_{\mathbf{j}}\circ D_{\mathbf{i}}=\sum_{\mathbf{k}}c_{\mathbf{i},\mathbf{j}}^{\mathbf{k}}D_{\mathbf{k}},\ \ %D_{\mathbf{i}}^{(p)}=\sum_{\mathbf{k}}d_{\mathbf{i},\mathbf{k}}D_{\mathbf{k}}$$
%for $\mathbf{i},\mathbf{j}$ consisting of powers of $p$ and $0$, then $\mathbb{D}$ is $\mathfrak{g}$-iterative.
%\end{lemma}
%\begin{proof}
%Define $d_{\mathbf{i},\mathbf{k}}$ earlier!!! Idea: it is enough for $m=1$ from Remark \ref{plie}, so should be enough in general ????????????????????????
%\end{proof}
\noindent
We comment below on a related notion of a restricted Lie algebra action (see \cite{Tyc}).
\begin{remark}\label{plie}
For $m=1$, any finite group scheme of the form considered in this paper (i.e. any finite group scheme of the \emph{Frobenius height one}) is equivalent to a \emph{restricted Lie algebra} in the sense of the theorem on page 139 of \cite{MuAbel}. Hence a $\mathfrak{g}$-derivation ($m=1$) on $R$ is equivalent to an action on $R$ of $e$ derivations satisfying the commutator and the $p$-th composition rules given by the corresponding restricted Lie algebra $\lie(\mathfrak{g})$ (see \cite{MuAbel}).
\end{remark}
\noindent
We need to know that the unique extension in Proposition \ref{etalehs} preserves the $\mathfrak{g}$-iterativity condition.
\begin{prop}\label{etalegext}
 Assume that $R\to S$ is  \'{e}tale and $\mathbb{D}$ is a $\mathfrak{g}$-derivation on $R$. Then the unique extension of $\mathbb{D}$ to $S$ in Proposition \ref{etalehs} is a $\mathfrak{g}$-derivation.
\end{prop}
\begin{proof}
The proof of the moreover part of \cite[theorem 27.2]{mat} may be applied here, similarly as in the proof of Proposition \ref{etalehs}.
\end{proof}
\begin{remark}\label{etalegextsch}
As before, Proposition \ref{etalegext} easily generalizes to $\mathfrak{g}$-derivations on $k$-schemes.
\end{remark}
\noindent
We prove a version of the ``Wronskian theorem" \cite[Thm. II.1]{kol1} for the case of $\mathfrak{g}$-derivations.
\begin{prop}\label{wronskian}
Let $\mathbf{K}$ be a $\mathfrak{g}[1]$-field and $C$ be its field of constants. Then for any positive integer $l$ and any $x_1,\ldots,x_l\in K$, the elements $x_1,\ldots,x_l$ are linearly independent over $C$ if and only if the rank of the following ``Wronskian matrix"
$$\big(D_{\mathbf{i}}(x_j)\big)_{\mathbf{i}\in [p]^e,j\leqslant l}$$
is strictly smaller than $l$.
\end{prop}
\begin{proof}
Assume that $x_l=c_1x_1+\ldots+c_{l-1}x_{l-1}$ for some $c_1,\ldots, c_{l-1}\in C$. By Lemma \ref{absconst}, each $D_{\mathbf{i}}$ is $C$-linear. Hence we obtain that the rank of our Wronskian matrix is smaller than $l$ as in the standard case (see \cite[Thm. II.1]{kol1}).
\\
Let the rank of the matrix $\big(D_{\mathbf{i}}(x_j)\big)_{\mathbf{i}\in [p]^e,j\leqslant l}$ be equal to $r<l$. After reordering $x_1,\ldots,x_l$, we may assume that the matrix $\big(D_{\mathbf{i}}(x_j)\big)_{\mathbf{i}\in [p]^e,j\leqslant r}$ has rank $r$. Therefore there exist $\lambda_1,\ldots, \lambda_{r+1}\in K$, not all equal to $0$, such that for each tuple $\mathbf{k}$ we have
\begin{equation}
\sum\limits_{s=1}^{r+1}\lambda_sD_{\mathbf{k}}(x_s)=0.\tag{$*$}
\end{equation}
\noindent
Reordering and dividing by $c_{r+1}$ if need be, we may assume that $c_{r+1}=1$. By Remark \ref{gderint2}, there are $c_{\mathbf{i},\mathbf{j}}^{\mathbf{k}}\in C$ such that for all tuples $\mathbf{i},\mathbf{j}$ we have
$$D_{\mathbf{j}}\circ D_{\mathbf{i}}=\sum_{\mathbf{k}}c_{\mathbf{i},\mathbf{j}}^{\mathbf{k}}D_{\mathbf{k}}.$$
Hence for $\mathbf{1}:=(1,0,\ldots,0)$ and every tuple $\mathbf{i}$ we get (using $(*)$ with $\mathbf{k}=\mathbf{1}$ for the last equality)
\begin{IEEEeqnarray*}{rCl}
0 & = & \sum\limits_{s=1}^{r+1}D_{\mathbf{1}}(\lambda_s)D_{\mathbf{i}}(x_s)+\sum\limits_{s=1}^{r+1}\lambda_sD_{\mathbf{1}}(D_{\mathbf{i}}(x_s))  \\
 & = & \sum\limits_{s=1}^{r}D_{\mathbf{1}}(\lambda_s)D_{\mathbf{i}}(x_s)+\sum\limits_{\mathbf{l}}c_{\mathbf{1},\mathbf{i}}^{\mathbf{l}}\sum\limits_{s=1}^{r+1}\lambda_sD_{\mathbf{l}}(x_s) \\
 & = & \sum\limits_{s=1}^{r}D_{\mathbf{1}}(\lambda_s)D_{\mathbf{i}}(x_s).
\end{IEEEeqnarray*}
Thus $D_{(1,0,\ldots,0)}(\lambda_s)=0$ for every $s\leqslant r$. After similar reasoning for the derivations $D_{(0,1,0,\ldots,0)},\ldots,D_{(0,\ldots,0,1)}$ we get $\lambda_1,\ldots,\lambda_{r+1}\in C$. By setting $\mathbf{k}=\mathbf{0}$ in $(*)$, we obtain that $x_1,\ldots,x_l$ are linearly dependent over $C$.
\end{proof}
\noindent
The next result generalizes the first part of \cite[Lemma 2.1]{Zieg2} (the second part is generalized in Proposition \ref{constantdisjoint}(1)). For a group scheme action interpretation and more comments, see Remark \ref{actioncom}.
\begin{cor}\label{dimconstants}
Let $\mathbf{K}$ be a $\mathfrak{g}$-field and $C$ its field of constants. Then we have
$$[K:C]\leqslant p^e.$$
\end{cor}
\begin{proof}
Let $l=p^e+1$ and $x_1,\ldots,x_l\in K$. The rank of the corresponding Wronskian matrix from Proposition \ref{wronskian} is at most $p^e$, since there are only $p^e$ operators of the form $D_{\mathbf{i}}$. By Proposition \ref{wronskian}, $x_1,\ldots,x_l$ are linearly dependent over $C$.
\end{proof}
\noindent
We can generalize now the appropriate results from \cite{Zieg2} to the context of $\mathfrak{g}$-derivations. Let $K$ be a field of characteristic $p$. Then $[K:K^p]=p^l$, where $l\in \Nn\cup \{\infty\}$. We call $l$ the \emph{degree of imperfection} of $K$. For the definition and properties of separable algebras/field extensions the reader is referred to \cite[Sect. 26]{mat}. The next result is a crucial characterization of the \'{e}tale extensions which we will need for the quantifier elimination results in Section \ref{secmc}.
\begin{lemma}\label{constantdisjoint}
Let $\mathbf{K}\subseteq \mathbf{L}$ be an extension of $\mathfrak{g}$-fields. Let $C_K$ (resp. $C_L$) be the constant field of $\mathbf{K}$ (resp. $\mathbf{L}$). Then we have the following.
\begin{enumerate}
\item The field $K$ is linearly disjoint from $C_L$ over $C_K$.

\item If $\mathbf{K}$ is strict (see Def. \ref{firstdefstrict}(2)), then the extension $K\subseteq L$ is separable.

\item If the extension $K\subseteq L$ is \'{e}tale and $K$ has a finite degree of imperfection, then $\mathbf{K}$ is strict if and only if $\mathbf{L}$ is strict.
\end{enumerate}
\end{lemma}
\begin{proof}
For $(1)$, we apply Lemma \ref{wronskian} as in the standard case (see  \cite[Cor. 1, p. 87]{kol1}).
\\
The item $(2)$ follows directly from $(1)$ using \cite[Thm. 26.4]{mat}.
\\
The right-to-left implication in $(3)$ is clear (and only the condition $L^p\cap K=K^p$ is used). For the left-to-right implication, the \'{e}tale assumption implies that $[L:L^p]=[K:K^p]$ and $KL^p=L$. Since $L^p\subseteq C_L$, by $(1)$ (used for the second equality below) we get
$$[L:L^p]=[K:K^p]=[KC_L:C_L]=[L:C_L].$$
Since $[L:L^p]$ is finite, we get $C_L=L^p$.
\end{proof}

\subsection{Formal group actions}
Recall that $F$ is a formal group law over $k$ which may be identified with a direct system of finite group schemes over $k$ and $G$ is an algebraic group over $k$.
\begin{definition} We define the following.
\begin{enumerate}
\item An \emph{$F$-derivation} on $V$ is a direct system of $F[m]$-derivations on $V$.

\item A \emph{$G$-derivation} is a $\widehat{G}$-derivation (see Section \ref{gpschemetrun} for the definition of $\widehat{G}$).

\item Similarly we get the notions of an $F$-derivation and a $G$-derivation on $R$.
\end{enumerate}
\end{definition}
\begin{remark}
\begin{enumerate}
\item As in Remarks \ref{gderint}, \ref{gderint2}, any $F$-derivation is an $e$-dimensional HS-derivation which satisfies the \emph{$F$-iterativity law}.

%\item Any $F$-derivation $\mathbb{D}$ gives a system of $F[m]$ derivations $\mathbb{D}[m]$....

\item Note that a $G$-derivation on $V$ (i.e. a direct system of group scheme actions of $G[m]$ on $V$) is \emph{not} the same as an algebraic action of $G$ on $V$. Clearly, any algebraic action of $G$ on $V$ gives a $G$-derivation by restricting the action of $G$ to $G[m]$ for each $m$. The difference between these two notions is easy to observe for $R=k[t]$ and $G=\ga$. If $\mathbb{D}$ is a $\ga$-derivation on $R$ (i.e. an iterative HS-derivation on $R$), then $\mathbb{D}$ comes from a $\ga$-action on $\spec(R)$ if and only if there is $n$ such that for all $i>n$, we have $D_i(t)=0$.

\end{enumerate}
\end{remark}
%\noindent
%\begin{definition}
%Let $\mathbb{D}$ be an $F$-derivation on $R$.
%\begin{enumerate}
%\item The ring of constants of  $\mathbb{D}$ is he ring of constants of  $\mathbb{D}[1]$.
%\item We call $\mathbb{D}$ strict, if $\mathbb{D}[1]$ is strict.
%\end{enumerate}
%\end{definition}
\begin{example}\label{canex}
Let $R=k\llbracket \mathbf{X}\rrbracket$ and $K=k((\mathbf{X}))$. Similarly as in \cite[Section 3.2]{HK1}, we define a \emph{canonical $F$-derivation} on $R$ and $K$. As a $k$-algebra map, it is defined on $R$ as follows
$$\mathbb{D}^F=\ev_F:R\to R\llbracket Y_1,\ldots,Y_e\rrbracket,\ \ \mathbb{D}^F(f)=f(F).$$
By Proposition \ref{etalegext}, $\mathbb{D}^F$ uniquely extends to an $F$-derivation on $K$ which we also call canonical and also denote by $\mathbb{D}^F$. For any $m$, we call $\mathbb{D}^F[m]$ (on $R$ or on $K$) a \emph{canonical $F[m]$-derivation}.
\end{example}
\noindent
We point out here that for a given $\mathfrak{g}$, there may exist non-isomorphic formal groups $F_1,F_2$ such that $F_1[m]=\mathfrak{g}=F_2[m]$.
\begin{prop}\label{canstrict}
The canonical $F$-derivation (equivalently, $F[m]$-derivation) is strict.
\end{prop}
\begin{proof}
We introduce the following notation:
$$\partial^F_1=D_{(1,0,\ldots,0)},\ldots,\partial^F_e=D_{(0,\ldots,0,1)}.$$
Let $i\leqslant e$ and $f\in K$. By the chain rule, we get
\begin{equation}
\partial^F_i(f)=\Big(\frac{\partial F_i}{\partial Y_j}(\mathbf{X},0)\Big)_{i,j}\cdot \Big(\frac{\partial f}{\partial X_1},\ldots,\frac{\partial f}{\partial X_e}\Big).\tag{$*$}
\end{equation}
Let $J$ denote the Jacobian matrix appearing in $(*)$. Then $J$ is the matrix of the derivative (at $\mathbf{0}$) of the formal map which is the ``$F$-multiplication by $\mathbf{X}$''. Since this formal map is (formally) invertible, the matrix $J$ is non-singular. Therefore $f$ belongs to the field of constants if and only if $\frac{\partial f}{\partial X_i}=0$ for each $i$. The latter condition occurs if and only if $f$ is a $p$-th power, so the result is proved.
\end{proof}
\noindent
In the case when $F=\widehat{G}$, we can define a canonical $F$-derivation on a localization of a $k$-algebra of finite type.
\begin{example}\label{cangex}
Let $\mathcal{O}_G$ be the local ring of $G$ at the identity. Since $G$ is a smooth variety over $k$, $\mathcal{O}_G$ is a regular local ring. Let $\mathbf{x}=(x_1,\ldots,x_e)$ be a sequence of local parameters in $\mathcal{O}_G$. By \cite[Thm 30.6(i)]{mat}, the ring $\widehat{\mathcal{O}}_G$ is the power series ring in the variables $\mathbf{x}$. If $F=\widehat{G}$, then $F(\mathbf{x},\mathbf{Y})\in \mathcal{O}_G\llbracket \mathbf{Y}\rrbracket$ (the group action is algebraic!). Hence $\mathcal{O}_G$ is a $G$-subring (after identifying $\mathbf{x}$ with $\mathbf{X}$) of $(k\llbracket \mathbf{X}\rrbracket,\mathbb{D}^F)$. Therefore the field of rational functions $k(G)$ has a natural $G$-derivation on it which we call the \emph{canonical  $G$-derivation} on $k(G)$.
%is just the regular action of $G$ on $G$ (by left translations). Since open immersions are \'{e}tale, using Remark \ref{etalegextsch} we get a canonical %$G$-derivation on each open subvariety $U\subseteq G$. Thus we also have a canonical $G$-derivation on the field of rational functions $k(G)$.
We also get a canonical $G[m]$-derivation on $k(G)$.
\\
If $G$ is affine, then we also have a canonical $G$-derivation on $k[G]$. The natural extensions
$$k[G]\subseteq K\llbracket \mathbf{X}\rrbracket,\ \ \ k(G)\subseteq K((\mathbf{X})),$$
where the local parameters on $G$ are understood as variables as in \cite[Thm 30.6(i)]{mat}, are $G$-extensions by our construction.
\end{example}

\subsection{Strict $\mathfrak{g}$-derivations and group scheme actions}\label{secstr}
In this subsection we will investigate strict $\mathfrak{g}$-rings using group scheme actions. Let us fix $\mathbb{D}$, a $\mathfrak{g}$-derivation on $V$.
For the notion of a (free) action of a group scheme on a scheme and its (good) quotient, the reader is advised to consult Section 12 of \cite{MuAbel}.
%If $V=\spec(R)$, then let us define the $k$-algebra of \emph{constants} of $\mathbb{D}$ as
%$$C:=\bigcap_{\mathbf{i}\neq (0,\ldots,0)} \ker(D_{\mathbf{i}}).$$
\begin{theorem}\label{quotient}
We have the following.
\begin{enumerate}
\item The quotient scheme $V/\mathfrak{g}$ exists.

\item If $V=\spec(R)$, then $V/\mathfrak{g}=\spec(C_{\mathbb{D}})$, where $C_{\mathbb{D}}$ is the ring of absolute constants of $\mathbf{R}$ (see Definition \ref{firstdefstrict}(3)).
\end{enumerate}
\end{theorem}
\begin{proof}
For $(1)$, we quote \cite[Thm 1(A), p. 111]{MuAbel}.
\\
By the proof of \cite[Thm 1(A), p. 111]{MuAbel} we have $V/\mathfrak{g}=\spec(C')$ where
$$C'=\{r\in R\ |\ \mathbb{D}(r)=r\}.$$
Clearly $C'$ coincides with $C_{\mathbb{D}}$ giving $(2)$.
\end{proof}

\begin{remark}\label{actioncom}
Let $\mathbb{D}$ be a $\mathfrak{g}[1]$-derivation on a field $K$ with the field of constants $C$. Using Proposition \ref{wronskian}, it is easy to see that if $x_1,\ldots,x_n\in K$ are linearly independent over $C$, then $\mathbb{D}(x_1),\ldots,\mathbb{D}(x_n)\in K[\mathbf{v}_1]$ are linearly independent over $K$. Therefore the induced $K$-linear map
$$\widetilde{\mathbb{D}}:K\otimes_CK\to K[\mathbf{v}_1],\ \ \ \ \ \ \widetilde{\mathbb{D}}(a_1\otimes b_1+\ldots+a_n\otimes b_n)=a_1\mathbb{D}(b_1)+\ldots+a_n\mathbb{D}(b_n)$$
is an embedding. Since we have
$$\dim_KK\otimes_CK=[K:C]\leqslant p^e,\ \ \ \dim_KK[\mathbf{v}_1]=p^e,$$
the following are equivalent:
\begin{enumerate}
\item $[K:C]=p^e$,

\item the map $\widetilde{\mathbb{D}}$ is onto,

\item the map $\widetilde{\mathbb{D}}$ is an isomorphism.
\end{enumerate}
\noindent
In terms of group scheme actions, the above equivalences mean that the action of $\mathfrak{g}$ on $\spec(K)$ is free if and only if $[K:C]=p^e$, and if this action is free, then the corresponding quotient $\spec(C)=\spec(K)/\mathfrak{g}$ is a good quotient. Note that the general theorem about group scheme actions \cite[Thm 1(B), p. 112]{MuAbel} gives the left-to-right implication above.
%Corollary \ref{dimconstants} gives the inequality for any (i.e. not necessarily strict) $\mathfrak{g}$-field.
%Perhaps \cite[Thm 1(B), p. 112]{MuAbel} generalizes to a similar inequality for arbitrary actions (not necessarily giving good quotients)? Intuitively it should %follow from the surjectivity of the map
%$$G\times X\to X\times_{X/G}X,$$
%which is clear in the category of sets.
\end{remark}
\noindent
We comment below on an interpretation of the notion of strictness using group scheme actions.
\begin{remark}\label{factorstrict}
Let $V^{\fr^m}$ be $V$ twisted by the $m$-th power of the Frobenius automorphism as in Section \ref{gpschemetrun}. From the universal property of quotients, there is a unique morphism $\Psi$ making the following diagram commutative
\begin{equation*}
\xymatrixcolsep{4pc}
\xymatrix{V \ar[d]_{\fr^m_V} \ar[r]^{}& V/\mathfrak{g} \ar[dl]_{\Psi}\\
	   V^{\fr^m} .  }
\end{equation*}
If $V=\spec(R)$ and $m=1$, then $\mathbf{R}$ is strict if and only if $\Psi$ is an isomorphism. It is also easy to see that for a reduced $R$ and arbitrary $m$,  $\mathbf{R}$ is strict if and only if $\Psi$ is an isomorphism.
%Hence we can extend the definition of strictness to $\mathfrak{g}$-derivations on arbitrary schemes over $k$.
\end{remark}
\noindent
We will need the following result in Section \ref{sececf}.
\begin{lemma}\label{deronc}
Let $\mathbb{D}$ be a $\mathfrak{g}$-derivation on $R$ and $C$ be the ring of constants of $(R,\mathbb{D})$. Then we have the following.
\begin{enumerate}
\item $C$ is a $\mathfrak{g}$-subring.

\item A $\mathfrak{g}$-action on $\spec(C)$ naturally induces a $\mathfrak{g}[m-1]^{\fr}$-action on $\spec(C)$, hence $C$ is naturally a $\mathfrak{g}[m-1]^{\fr}$-ring.
\end{enumerate}
\end{lemma}
\begin{proof}
We work on the level of group scheme actions. By Theorem \ref{quotient}(2), we have $\spec(C)=\spec(R)/\mathfrak{g}[1]$. As in the case of the usual group actions (one can work it out on the level of rational points), we get the induced action of $\mathbb{D}$ on $\spec(C)$ (since $\mathfrak{g}[1]$ is normal in $\mathfrak{g}$) giving $(1)$.

For $(2)$, a similar argument gives a natural action of $\mathfrak{g}/\mathfrak{g}[1]$ on $\spec(C)$. By Remark \ref{gm}, we have
$$\mathfrak{g}[m-1]^{\fr}\cong \mathfrak{g}/\mathfrak{g}[1],$$
which proves $(2)$.
\end{proof}
\begin{remark}\label{dpidpj}
We can describe the $\mathfrak{g}[m-1]^{\fr}$-action on $\spec(C)$ from Lemma \ref{deronc}(2) more specifically, since it is given by
$(D_{p\mathbf{j}}|_C)_{\mathbf{j}\in [p^{m-1}]^{\times e}}$. Hence for any $c\in C$ and $\mathbf{i},\mathbf{j}\in [p^{m-1}]^{\times e}$, we get the following:
$$D_{p\mathbf{j}}(D_{p\mathbf{i}}(c))=\sum_{\mathbf{k}\in [p^{m-1}]^{\times e}}\left(c_{\mathbf{i},\mathbf{j}}^{\mathbf{k}}\right)^pD_{p\mathbf{k}}(c).$$
\end{remark}

%Below we point out the connections between strictness and good quotients. It will not be used in this paper.
%\begin{remark}\label{good}
%Suppose $\mathbf{K}$ is a $\mathfrak{g}$-field. Then the $\mathfrak{g}$-action on $\spec(K)$ is free, if and only $\mathbf{K}$ is strict. Moreover, if %$\mathbf{K}$ is strict then the resulting free action gives a good quotient if and only if the degree of imperfection of $K$ is $e$.
%\end{remark}
%\begin{proof}
%The action is free if and only if $\mathbb{D}(R)$ generates $R[\mathbf{v}_m]$ as an $R$-algebra. Then do first the case of $e=1$ and then see that the condition
%$$C\varsupsetneq \ker(\partial_1)\varsupsetneq \ker(\partial_1)\cap \ker(\partial_2) \varsupsetneq \ker(\partial_1)\cap \ker(\partial_2) \cap %\ker(\partial_3)\varsupsetneq \ldots $$
%is necessary.
%\end{proof}
%\noindent
%The next result {\bf follows generally for any finite group scheme action???}.
\noindent
The next result (Proposition \ref{gstrictexist}) is rather easy in the case of full HS-derivations and turned out to be quite problematic for us in the restricted case. The proof follows the lines of the proof of \cite[Fact 2.5]{K3}, but, in the general case here, a different set of values of the highest order operators needs to be taken (see Remark \ref{values}). Firstly, we need a general lemma.
\begin{lemma}\label{hsideal}
For $i\leqslant e$ and $j<m$, we define $(p^j)_i\in [p^m]^{\times e}$ as a sequence consisting of zeroes, except for the $i$-th coordinate where it has $p^j$. Assume $R$ is a $\mathfrak{g}$-ring and let $\mathcal{P}$ be a subset of $[p^m]^{\times e}$ such that
$$\{(p^j)_i\ |\ i\leqslant e,j<m\}\subseteq \mathcal{P}.$$
Then for any $B\subset R$, the ideal generated by the set
$$B_{\mathcal{P}}=\{D_{\mathbf{k}}(b)\ |\ b\in B,\mathbf{k}\in \mathcal{P}\}$$
is an HS-ideal.
\end{lemma}
\begin{proof}
It follows from Remark \ref{gderint2} and Lemma \ref{newton}.
\end{proof}
\begin{prop}\label{gstrictexist}
Any $\mathfrak{g}$-field $\mathbf{K}$ has a strict $\mathfrak{g}$-field extension.
\end{prop}
\begin{proof}
We consider the theory of $\mathfrak{g}$-fields as a universal theory in the language of $\mathfrak{g}$-fields (containing $-$ and $\div$). Then any  $\mathfrak{g}$-field embeds into an existentially closed $\mathfrak{g}$-field. Therefore, it is enough to prove that existentially closed $\mathfrak{g}$-fields are strict, so we may assume that $\mathbf{K}$ is existentially closed. Assume that $\mathbf{K}$ is not strict, i.e. there is $a\in K\setminus K^p$ which is a constant of $\mathbf{K}$. To reach a contradiction (with the assumption that $\mathbf{K}$ is existentially closed), it is enough to find a $\mathfrak{g}$-extension $\mathbf{K}\subseteq \mathbf{L}$ such that $a^{1/p}\in L$. Let $C$ denote the field of constants of $\mathbf{K}$ and $B$ be a $p$-basis of $C$ over $K^p$ such that $a\in B$.
\\
\\
{\bf Claim 1}
\\
There is a $\mathfrak{g}[m-1]$-derivation on $C^{1/p}$ extending the one we have on $K$ and such that for each $b\in B$ and each $\mathbf{j}\in [p^{m-1}]^{\times e}$, we have
$$D_{\mathbf{j}}(b^{1/p})=(D_{p\mathbf{j}}(b))^{1/p}.$$
\begin{proof}[Proof of Claim 1]
By Lemma \ref{deronc}(2) and Remark \ref{dpidpj},
$$\mathbb{D}':=\left(D_{p\mathbf{j}}|_C\right)_{\mathbf{j}\in [p^{m-1}]^{\times e}}$$
is a $\mathfrak{g}[m-1]^{\fr}$-derivation on $C$. Let $\fr_C^{-1}:C\cong C^{1/p}$, and $\mathbb{D}''$ be $\mathbb{D}'$ transported to $C^{1/p}$ using $\fr_C^{-1}$. Then $\mathbb{D}''$ is a $\mathfrak{g}[m-1]$-derivation on $C^{1/p}$ and by the construction it has the required properties (see also Lemma \ref{transport} and the proof of Proposition \ref{strictext}).
\end{proof}
We consider the following rings (the set $B$ indexes the variables):
$$R:=K\left[X_b^{(\mathbf{i})}\ |\ \mathbf{i}\in [p^m]^{\times e},\ b\in B\right],$$
$$R':=K\left[X_b^{(\mathbf{j})}\ |\ \mathbf{j}\in [p^{m-1}]^{\times e},\ b\in B\right];$$
where for each $b\in B$, $X_b$ is identified with $X_b^{(0,\ldots,0)}$. We put a $\mathfrak{g}$-ring structure on $R$ which is $\mathfrak{g}$-extending $\mathbf{K}$ in the following way:
$$D_{\mathbf{j}}\left(X_b^{(\mathbf{i})}\right):=\sum_{\mathbf{f}\in [p^m]^{\times e}}c^{\mathbf{f}}_{\mathbf{i},\mathbf{j}}X_b^{(\mathbf{f})}.$$
Then $R'$ is a $\mathfrak{g}[m-1]$-subring of $R$.
%Let $I$ be the $\mathfrak{g}[m]$-ideal of $R$ which is HS$[m]$-generated by the following set
%$$\{X_b^p-b\ |\ b\in B\}.$$
%As a usual ideal, $I$ is generated by the following set
Let us define a subset $W\subset R'$ as follows:
$$W:=\left\{\left(X_b^{(\mathbf{j})}\right)^p-D_{p\mathbf{j}}(b)\ |\ \mathbf{j}\in [p^{m-1}]^{\times e},\ b\in B\right\}.$$
We define the following $K$-algebra map
$$\Psi:R'\to C^{1/p},\ \ \Psi\left(X_b^{(\mathbf{j})}\right)=D_{p\mathbf{j}}(b)^{1/p}$$
and let $\mathfrak{m}=\ker(\Psi)$.
By Claim 1, $\Psi$ is a $\mathfrak{g}[m-1]$-map, so $\mathfrak{m}$ is a maximal $\mathfrak{g}[m-1]$-ideal of $R'$ containing $W$. We will show that the $\mathfrak{g}$-ideal $J$ in $R$ which is $\mathfrak{g}$-generated by $\mathfrak{m}$ is prime (then we can take $L$ as the field of fractions of $R/J$). Let us order the set $[p^{m-1}]^{\times e}$ as $\mathbf{k}_{(1)},\ldots,\mathbf{k}_{(p^{(m-1)e})}$ such that for $i\leqslant j$, we have $|\mathbf{k}_{(i)}|\leqslant |\mathbf{k}_{(j)}|$ (in particular $\mathbf{k}_{(1)}=(0,\ldots,0)$). We also order $B=(b_s)_{s<\kappa}$ such that $b_0=a$.
\\
\\
{\bf Claim 2}
\\
There is a set of generators of $\mathfrak{m}$ consisting of elements of the following two types:
$$\left(X_b^{(\mathbf{j})}\right)^p-D_{p\mathbf{j}}(b),\ \ \ \ \ \ \ \ \ \ \ \ X_{b_s}^{(\mathbf{k}_{(j)})} - \sum_{i=1}^{j-1}\sum _{l=0}^{p-1}\alpha_{t,i,l}\left(X_{b_s}^{(\mathbf{k}_{(i)})}\right)^l - \sum_{t<s}\sum_{\mathbf{j}\in [p^{m-1}]^{\times e}}\beta_{t,\mathbf{j}}X_{b_t}^{(\mathbf{j})};$$
where $b\in B$, $s<\kappa$, $\mathbf{j}\in [p^{m-1}]^{\times e}$, $j\in \{1,\ldots,p^{(m-1)e}\}$ and $\alpha_{i,l},\beta_{t,\mathbf{j}}\in K$.
\begin{proof}[Proof of Claim 2]
We construct a required set of generators in $p^{(m-1)e}\cdot \kappa$ steps.

In Step (0,1), we add $X^p-b_0$ to the (so far empty) set of generators.

In Step (0,2), we consider two cases.
\\
Case 1: $D_{p\mathbf{k}_{(2)}}(b_0)\notin K^p(b_0)$.
\\
In this case, we add the element
$$\left(X_{b_0}^{(\mathbf{k}_{(2)})}\right)^p-D_{p\mathbf{k}_{(2)}}(b_0)$$
to the set of generators.
\\
Case 2: $D_{p\mathbf{k}_{(2)}}(b_0)\in K^p(b_0)$.
\\
In this case
$$D_{p\mathbf{k}_{(2)}}(b_0)=\sum_{l=0}^{p-1}\alpha_l^pb_0^l,$$
for some $\alpha_0,\ldots,\alpha_{p-1}\in K$, and we add the element
$$X_{b_0}^{(\mathbf{k}_{(2)})}-\sum_{l=0}^{p-1}\alpha_l^pX_{b_0}^l$$
to the set of generators.

In Step (0,3), we proceed as in Step (0,2) with the field $K^p(b_0,D_{p\mathbf{k}_{(2)}}(b_0))$ replacing the field $K^p(b_0)$.

Continuing like this, after $p^{(m-1)e}$ steps we obtain our desired set of generators for $b_0$. We continue in a similar way (although e.g. Step (1,1) already consist of two cases) for $(b_i)_{i>0}$.
\end{proof}
By Lemma \ref{hsideal}, it is enough to apply to the generators from Claim 2 the operators $D_{\mathbf{i}}$, where $\mathbf{i}\in [p^{m-1}]^{\times e}$ or where $\mathbf{i}$ is of the form $(p^{m-1})_i$ (see Lemma \ref{hsideal} for the definition of $(p^{m-1})_i$). If $\mathbf{i}\in [p^{m-1}]^{\times e}$ then we get elements of $\mathfrak{m}$, so we may focus on the other case. To treat the first type of the generators from Claim 2, we just need to use that there is $\mathbf{l}\in [p^{m-1}]^{\times e}$ such that $\mathbf{i}=p\mathbf{l}$.
%By the construction of $\mathfrak{m}$,
We obtain the following:
\begin{IEEEeqnarray*}{rCl}
D_{p\mathbf{l}}\left(\left(X_b^{(\mathbf{j})}\right)^p-D_{p\mathbf{j}}(b)\right) & = & \left(D_{\mathbf{l}}\left(X_b^{(\mathbf{j})}\right)\right)^p-D_{p\mathbf{l}}\left(D_{p\mathbf{j}}(b)\right)  \\
 & = & \left(\sum_{\mathbf{k}\in [p^{m-1}]^{\times e}} c_{\mathbf{i},\mathbf{j}}^{\mathbf{k}}X_b^{(\mathbf{k})}\right)^p  -
 \sum_{\mathbf{k}\in [p^{m-1}]^{\times e}} (c_{\mathbf{i},\mathbf{j}}^{\mathbf{k}})^pD_{p\mathbf{k}}(b)
   \\
 & = & \sum_{\mathbf{k}\in [p^{m-1}]^{\times e}} \left(c_{\mathbf{i},\mathbf{j}}^{\mathbf{k}}\right)^p\left(\left(X_b^{(\mathbf{k})}\right)^p-D_{p\mathbf{k}}(b)\right),
\end{IEEEeqnarray*}
where the second equality holds by Remark \ref{dpidpj}, since $b\in C$. Hence we get that $D_{p\mathbf{l}}((X_b^{(\mathbf{j})})^p-D_{p\mathbf{j}}(b))$ belongs to the ideal generated by $W$ in $R'$, so it also belongs to $\mathfrak{m}$.

%For $\mathbf{i}$ such that for each $l\leqslant e$, we have $\mathbf{i}_l\in \{0,p^{m-1}\}$, let us define
%$$[\mathbf{i}] ...?????$$
After applying $D_{\mathbf{i}}$ to the second type of generators, by Lemma \ref{newton} we get the following:
$$D_{\mathbf{i}}\left(   X_{b_s}^{(\mathbf{k}_{(j)})} - \sum_{i=1}^{j-1}\sum_{l=0}^{p-1}\alpha_{i,l}\left(X_{b_s}^{(\mathbf{k}_{(i)})}\right)^l - \sum_{t<s}\sum_{\mathbf{j}\in [p^{m-1}]^{\times e}}\beta_{t,\mathbf{j}}X_{b_t}^{(\mathbf{j})} \right)  = $$
$$ {\mathbf{k}_{(j)}+\mathbf{i}\choose \mathbf{i}}X_{b_s}^{(\mathbf{k}_{(j)}+\mathbf{i})} + \sum_{i=1}^{j-1}\sum_{l=0}^{p-1}d_{i,l}\left(X_{b_s}^{(\mathbf{k}_{(i)}+\mathbf{i})}\right)^l + Q,$$
 for some $d_{i,l}\in K$, where
$${\mathbf{i}+\mathbf{j}\choose \mathbf{i}}:={i_1+j_1 \choose i_1}\cdot \ldots \cdot {i_e+j_e \choose i_e},$$
and where
$$Q\in K\left[X_{b_t}^{(\mathbf{r})},\ X^{(\mathbf{j})}_{b_s}\ |\ t<s,\ \mathbf{r}\in [p^m]^{\times e},\ \mathbf{j}\in [p^{m-1}]^{\times e}\right].$$
We forget now about the HS-differential structure on $R$ and $J$. The obtained set of generators of $J$ is of the form $J_1\cup J_2$, where $J_1\subseteq \mathfrak{m}$ and $J_2\subseteq R'[X_{\gamma}]_{\gamma\in \mathcal{I}}$ (after renaming variables), and
$$\mathcal{I}:=\kappa\times \{1,\ldots,e\}\times (\mathbf{k}_{i})_{i\in \{1,\ldots,p^{(m-1)e}\}}$$
with the lexicographical order. Moreover, there is a subset $\mathcal{I}_0\subset \mathcal{I}$ such that without loss of generality (since $p$ does not divide ${\mathbf{k}_{(j)}+\mathbf{i}\choose \mathbf{i}}$), we have
$$J_2=\{X_{\gamma}-Q_{\gamma}\ |\ \gamma\in \mathcal{I}_0\},$$
where for each $\gamma\in \mathcal{I}_0$, $Q_{\gamma}\in R'[X_{\delta}]_{\delta<\gamma}$. Therefore, the ring $R/J$ is isomorphic to a polynomial ring over $R'/\mathfrak{m}$, so the ideal $J$ is prime.
\end{proof}
\begin{remark}\label{values}
\begin{enumerate}
\item The proof above may seem to be technical, but the idea is very simple. For $\mathbf{j}\in [p^{m-1}]^{\times e}$, the value $D_{\mathbf{j}}(a^{1/p})$ is determined by Claim 1. For $i\leqslant e$, we set $D_{(p^{m-1})_i}(a^{1/p})$ as a new formal variable. The values of the remaining operators on $a^{1/p}$ are determined by the $\mathfrak{g}$-iterativity rule.

\item 

\item The above proof may be also used to show that a $\mathfrak{g}$-field is strict if and only if it is \emph{differentially perfect}, i.e. each of its differential extension is separable. This generalizes a result of Kolchin (see Chapter II, Section 3., Proposition 5.(a) in \cite{kol1}).
\end{enumerate}
\end{remark}
\noindent
We need one more lemma for the proof of the amalgamation property for a class of $\mathfrak{g}$-fields.
\begin{lemma}\label{forap}
Assume that $\mathbf{R}$ and $\mathbf{S}$ are $\mathfrak{g}$-rings. Then there is a unique $\mathfrak{g}$-ring structure on $R\otimes_kS$ such that the natural maps $\mathbf{R}\to \mathbf{R}\otimes_k\mathbf{S}$, $\mathbf{S}\to \mathbf{R}\otimes_k\mathbf{S}$ are $\mathfrak{g}$-homomorphisms.
\end{lemma}
\begin{proof}
It is convenient to show a more general result. Assume that $V,W$ are schemes over $k$ and $\mathbb{D}_V$ and $\mathbb{D}_W$ are $\mathfrak{g}$-derivations on $V$ and $W$ respectively. As in the case of the usual group actions, we get a unique $\mathfrak{g}$-derivation on $V\times W$ such that the projections $V\times W\to V,W$ are $\mathfrak{g}$-invariant. By considering the affine schemes and dualizing, we get what we need.
\end{proof}
\begin{prop}\label{lambdaap}
Let $\mathbf{K}\subseteq \mathbf{L}_1,\mathbf{K}\subseteq \mathbf{L}_2$ be $\mathfrak{g}$-field extensions and assume that $\mathbf{K}$ is strict. Then $\mathbf{L}_1$ and $\mathbf{L}_2$ can be $\mathfrak{g}$-amalgamated over $\mathbf{K}$ into a $\mathfrak{g}$-field.
\end{prop}
\begin{proof}
By Proposition \ref{etalegext}, we can assume that
$L_1,L_2$ are separably closed (since separable algebraic extensions are \'{e}tale). Then $K^{\sep}$ is a subfield of $L_1$ and $L_2$. By Remark \ref{defgen}(5), it is a $\mathfrak{g}$-subfield. By Lemma \ref{constantdisjoint} (using Corollary \ref{dimconstants}), the $\mathfrak{g}$-field structure on $K^{\sep}$ is strict. Therefore we can assume that $K$ is separably closed as well.
\\
The proof will be finished exactly as in \cite[Prop. 2.6]{Zieg2}. We can assume that $L_1,L_2$ are subfields of a big field $\Omega$ and that they are algebraically disjoint over $K$. By Lemma \ref{constantdisjoint}, the extension $K\subseteq L_1$ is separable. Since $K$ is separably closed, this extension is regular (see \cite[p. 367]{lang2002algebra}). By \cite[Thm. VIII 4.12]{lang2002algebra}, $L_1$ and $L_2$ are linearly disjoint over $K$, thus $L_1\otimes_KL_2$ is a domain. By Lemma \ref{forap}, there is a $\mathfrak{g}$-structure on $L_1\otimes_KL_2$ extending those on $L_1$ and $L_2$. By Proposition \ref{etalegext} again, we get the required $\mathfrak{g}$-structure on the field of fractions of $L_1\otimes_KL_2$.
\end{proof}

\section{Model companions}\label{secmc}
\noindent
In this section we turn our attention to the model-theoretic properties of (truncated) HS-fields. Our results here generalize the corresponding results from \cite{Zieg2} and \cite{K3}, and our proofs do not differ much.

\subsection{Existentially closed $\mathfrak{g}$-fields}\label{sececg}
In this subsection, we assume that $\mathfrak{g}$ is of the form $W[m]$ for a formal group law $W$ (see Remark \ref{intgs}).
\\
\\
Let $K$ be a field of characteristic $p$ and $\lambda$ be the following \emph{$p$-th root function}:
$$\lambda:K\to K,\ \ \ \lambda(x)= \begin{cases} x^{1/p}\text{ for }x\in K^p, \\ 0\ \ \ \ \text{ for }x\notin K^p. \end{cases}$$
We introduce several languages.
\begin{itemize}
\item Let $L^{\lambda}$ be the language of rings expanded by a unary function symbol $\lambda$.

\item Let $L_{e,m}$ be the language of $k$-algebras (so there are constants in the language for the elements of $k$) with $m$-truncated $e$-dimensional HS-derivations.

\item Let $L^{\lambda}_{e,m}=L^{\lambda}\cup L_{e,m}$.
\end{itemize}
\noindent
\begin{lemma}\label{lamb}
Each field of characteristic $p$ has a natural $L^{\lambda}$-structure and we have the following.
\begin{enumerate}
\item A field extension $K\subseteq L$ is an $L^{\lambda}$-extension if and only if $L^p\cap K=K^p$.

\item If $\mathbf{K}\subseteq \mathbf{L}$ is an $L_{e,m}$-extension and $\mathbf{K}$ is strict, then $\mathbf{K}\subseteq \mathbf{L}$ is an $L^{\lambda}_{e,m}$-extension.

\item Suppose $R$ is a subring of a field $L$, $K$ is the field of fractions of $R$ and $R\subseteq L$ is an $L^{\lambda}$-extension. Then $K\subseteq L$ is an $L^{\lambda}$-extension.
\end{enumerate}
\end{lemma}
\begin{proof}
The item $(1)$ is clear and $(2)$ follows from Lemma \ref{constantdisjoint}(2). The last item follows by an easy computation.
\end{proof}
\noindent
We need the following well-known description of elementary extensions of separably closed fields.
\begin{lemma}\label{elementaryetale}
Let us assume that $\mathbf{K}$ and $\mathbf{L}$ are $\mathfrak{g}$-fields such that $K$ and $L$ are separably closed and that $K$ has finite imperfectness degree. Then the following are equivalent.
\begin{enumerate}
\item The fields $K$ and $L$ have the same (absolute) $p$-basis.

\item The $K$-algebra $L$ is \'{e}tale.

\item The extension $K\subseteq L$ is elementary (in the language of rings).

\item The extension $\mathbf{K}\subseteq \mathbf{L}$ is $L^{\lambda}_{e,m}$-elementary.

\item The extension $\mathbf{K}\subseteq \mathbf{L}$ is $L_{e,m}$-elementary.
\end{enumerate}
\end{lemma}
\begin{proof}
By the proof of \cite[Theorem 2.1]{Del1} and by \cite[Claim 2.2]{Del1}, the extension $K\subseteq L$ is elementary if and only if $K$ and $L$ have the same (absolute) $p$-basis. By \cite[Theorem 26.7]{mat}, it happens if and only if this extension is \'{e}tale, so we get the equivalence of $(1)$, $(2)$ and $(3)$.
\\
Since both the $\lambda$-function and the $\mathfrak{g}$-derivation are defined (over the field of the $p^m$-th powers) using the field operations and the elements of a $p$-basis, we get the equivalence of $(3)$ with $(4)$ and $(5)$.
\end{proof}
\noindent
We introduce now several theories.
\begin{itemize}
\item Let $\gdf$ be the theory of $\mathfrak{g}$-fields in the language $L_{e,m}$.

\item Let $\gdf_{\lambda}$ be the theory of strict $\mathfrak{g}$-fields in the language $L_{e,m}^{\lambda}$ i.e. $\gdf_{\lambda}$ contains the natural interpretation of $\lambda$ and the following extra axiom:
$$\forall x \Big((x\neq 0\ \wedge \ \lambda(x)=0)\ \rightarrow\ (D_{(1,0,\ldots,0)}(x)\neq 0\ \vee\ \ldots \vee D_{(0,\ldots,0,1)}(x)\neq 0)\Big).$$

\item Let $\gdcf$ be the theory $\gdf$ with the extra axioms for strict $\mathfrak{g}$-fields, and for separably closed fields of imperfectness degree $e$.

\item Similarly for $\gdcf_{\lambda}$.
\end{itemize}
\noindent
\begin{lemma}\label{consistent}
The theories $\gdcf$ and $\gdcf_{\lambda}$ are consistent.
\end{lemma}
\begin{proof}
We take the canonical $F$-derivation $\mathbb{D}^F$ on the field $K$ from Example \ref{canex}. By Proposition \ref{canstrict}, $\mathbf{K}$ is strict and clearly the imperfection degree of $K$ is $e$. By Proposition \ref{etalegext}, $\mathbb{D}^F$ extends to the separable closure of $K$ (a separable algebraic extension is \'{e}tale, see \cite[Thm. 26.7]{mat}). By Lemma \ref{constantdisjoint}(3), this extension is still strict. Since it is \'{e}tale, the degree of imperfection  does not change.
\end{proof}
\noindent
The main technical result needed for quantifier elimination is the proposition below. It does not differ much from \cite[3.1]{Zieg2} where the case of full (i.e. non-truncated) HS-derivations is considered. For reader's convenience we include a proof following the lines of the proof of \cite[3.1]{Zieg2}.
\begin{prop}\label{elemamal}
Let us take $\mathbf{L}\models \gdcf_{\lambda}$, $\mathbf{F}$ being a strict $\mathfrak{g}$-field and assume that $\mathbf{K}$ is an  $L^{\lambda}_{e,m}$-substructure of both $\mathbf{F}$ and $\mathbf{L}$. Then there is an $L^{\lambda}_{e,m}$-embedding of $\mathbf{F}$ over $\mathbf{K}$ into an elementary extension of $\mathbf{L}$. \end{prop}
\begin{proof}
By Remark \ref{defgen}(1), Proposition \ref{etalegext} and Lemma \ref{lamb}(3), we can assume that $K$ is a field. By Lemma \ref{lamb}(1), we have $K^p=L^p\cap K$. Since $\mathbf{L}\models \gdcf_{\lambda}$, $\mathbf{L}$ is strict and hence $\mathbf{K}$ is strict as well.
By Proposition \ref{lambdaap} and Lemma \ref{lamb}(2), we can $L^{\lambda}_{e,m}$-amalgamate $\mathbf{F}$ and $\mathbf{L}$ over $\mathbf{K}$ into a strict $\mathfrak{g}$-field $\mathbf{F}'$. Since the extension $L\subseteq F'$ is separable, by \cite[Claim 2.2]{Del1} there is an elementary extension in the language of fields $L\preccurlyeq L'$ and a field embedding $\Upsilon:F'\to L'$ over $L$. The situation is described in the following diagram:
\begin{equation*}
\xymatrixcolsep{4pc}
\xymatrix{  &   &  L'\\
 &  \mathbf{F}'\ar[ru]^{\Upsilon} &  \\
\mathbf{F}\ar[ru]^{} &   & \mathbf{L}\ar[lu]^{} \ar[uu]^{\preccurlyeq}\\
  &  \mathbf{K}\ar[ru]^{}\ar[lu]^{}.  &  }
\end{equation*}
By Lemma \ref{elementaryetale}, the extension $L\subseteq L'$ is \'{e}tale. The $\mathfrak{g}$-field $\mathbf{F}'$ is strict, so by Corollary \ref{dimconstants}, we have $[F':(F')^p]\leqslant p^e$. Since the extension $L\subseteq F'$ is separable and $[L:L^p]=p^e$, by the equivalence between $(1)$ and $(2)$ in Lemma \ref{elementaryetale}, the extension $L\subseteq F'$ is \'{e}tale. Then the extension $F'\subseteq L'$ is  \'{e}tale as well. By Proposition \ref{etalegext}, there is a $\mathfrak{g}$-derivation on $L'$ extending the one on $F'$ (hence also extending the one on $L$). By Lemma \ref{elementaryetale}, the extension $\mathbf{L}\subseteq \mathbf{L}'$ is $L^{\lambda}_{e,m}$-elementary and we are done.
\end{proof}
\noindent
We are ready to prove the main result of this subsection.
\begin{theorem}\label{qegdcf}
We have the following.
\begin{enumerate}
\item The theory $\gdcf_{\lambda}$ has quantifier elimination in the language $L^{\lambda}_{e,m}$.

\item Each model of $\gdf$ embeds into a model of $\gdcf$.

\item The theory $\gdcf$ is a model companion of the theory $\gdf$.

\item The theory $\gdcf_{\lambda}$ is a model completion of the theory $\gdf_{\lambda}$.
\end{enumerate}
\end{theorem}
\begin{proof}
For $(1)$, we use the criterion from \cite[Theorem 13.1]{sacks} (and Proposition \ref{elemamal}) exactly as in \cite[3.1]{Zieg2}.
\\
For $(2)$, let us take any $\mathfrak{g}$-field $\mathbf{F}$. By Proposition \ref{gstrictexist}, we may assume that $\mathbf{F}$ is strict. By Lemma \ref{consistent}, there is $\mathbf{L}\models \gdcf_{\lambda}$. Clearly, $k$ with the trivial $\mathfrak{g}$-structure is a common $L^{\lambda}_{e,m}$-substructure of both $\mathbf{F}$ and $\mathbf{L}$. By Proposition \ref{elemamal}, $\mathbf{F}$ embeds into an elementary extension of $\mathbf{L}$, which is clearly a model of  $\gdcf_{\lambda}$ as well.
\\
By $(1)$, $\gdcf_{\lambda}$ is model complete, so using Lemma \ref{lamb}(2) and Lemma \ref{elementaryetale}, we get that $\gdcf$ is model complete as well. By $(2)$, $\gdcf$ is a model companion of the theory $\gdf$, so we get $(3)$.
\\
By Proposition \ref{lambdaap} and the item $(1)$, we get $(4)$.
\end{proof}
\begin{remark}
\begin{enumerate}
\item The theory $\gdf$ does not have the amalgamation property. The theory $\gdcf$ does not have quantifier elimination and the theory $\gdcf$ is not a model completion of the theory $\gdf$.

\item We have (after taking $k=\Ff_p$)
$$\ga^e[m]-\dcf =\sch_{p,e,m},$$
where $\sch_{p,e,m}$ is the theory considered in \cite{K3}.

\item If we take the algebraic group $U$ from Example \ref{gitex}(2), then a $U(1)$-field is the same as a field with two derivations $\partial_1,\partial_2$ such that $\partial_1^{(p)}=\partial_2$ and $\partial_2^{(p)}=0$. Hence $U(1)-\dcf$ corresponds to Wood's theory $2-\dcf$, see \cite{Wo2}. It should be possible to find algebraic groups governing the iterative rules for Wood's theories $m-\dcf$ for an arbitrary $m$.

\item After dropping the iterativity assumptions rather strange things happen. In the case of characteristic zero, model companions exist and are analyzed in \cite{MS2}. In the case of positive characteristic, it is shown in \cite[Prop 7.2]{MS2} that (in our terminology) the theory of fields with $m$-truncated $e$-dimensional HS-derivations has a model companion if and only if $m=1$.

\end{enumerate}
\end{remark}

\subsection{Existentially closed $F$-fields}\label{sececf}
Let $L_e$ be the language of $k$-algebras with $e$-dimensional HS-derivations. The main model-theoretic advantage of HS-derivations (over truncated HS-derivations) is that we do not need to consider the extra operator $\lambda$ to get quantifier elimination results.
\\
We define two $L_e$-theories.
\begin{itemize}
\item Let $\fdf$ be the theory of $F$-fields.

\item Let $\fdcf$ be the theory $\fdf$ with the extra axioms for strict $F$-fields, and for separably closed fields of imperfection degree $e$.
\end{itemize}
\noindent
The main algebraic difference between $\mathfrak{g}$-derivations and $F$-derivations is given by the proposition below which generalizes \cite[Lemma 2.4]{Zieg2}. First we need an obvious lemma which is a general fact about group scheme actions.
\begin{lemma}\label{transport}
Let $f\in \aut(k)$ and let $\varphi:R\to S$ be an isomorphism of rings extending $f$. For any $\mathfrak{g}$-derivation $\mathbb{D}$ on $R$, we define:
$$\mathbb{D}^{\varphi}:S\to S[\mathbf{v}_m],\ \ \ \mathbb{D}^{\varphi}:=\varphi[\mathbf{v}_m]\circ \mathbb{D}\circ \varphi^{-1}.$$
Then $\mathbb{D}^{\varphi}$ is a $\mathfrak{g}^{f}$-derivation, where $\mathfrak{g}^{f}$ is the group scheme $\mathfrak{g}$ twisted by $f$.
\end{lemma}
\begin{prop}\label{strictext}
 Let $\mathbf{K}=(K,\mathbb{D})$ be an $F$-field. Then there is a smallest strict $F$-field extending $\mathbf{K}$.
\end{prop}
\begin{proof}
Let $C$ be the field of constants of $\mathbf{K}$. It is enough to show that there is a unique $F$-derivation $\mathbb{D}'$ on $C^{1/p}$ extending $\mathbb{D}$. By Lemma \ref{deronc}(1), $C$ is an $F$-subfield of $K$. For $\mathbf{i}\in \Nn^e$ and $a\in C^{1/p}$, we have the only option (as in \cite[Lemma 2.4]{Zieg2}):
$$D'_{\mathbf{i}}(a):=(D_{p\mathbf{i}}(a^p))^{1/p}.$$
Clearly, each $D'_{\mathbf{i}}$ extends $D_{\mathbf{i}}$. We need to show that $\mathbb{D}'=(D'_{\mathbf{i}})_{\mathbf{i}\in \Nn^e}$ is an $F$-derivation. By Lemma \ref{deronc}(2), the sequence of maps
$$\mathbb{D}'':=(D_{p\mathbf{i}}(b))_{\mathbf{i}\in [p^m]^e}$$
is an $F[m-1]^{\fr}$-derivation on $C$. It is easy to see that $\mathbb{D}'[m-1]$ coincides with $\mathbb{D}''$ ``transported'' to $C^{1/p}$ using the ring isomorphism $\fr^{-1}_C:C\to C^{1/p}$. By Lemma \ref{transport}, $\mathbb{D}'[m-1]$ is an $F[m-1]$-derivation. Since it happens for all $m\in \Nn$, we get that $\mathbb{D}'$ is an $F$-derivation.
\end{proof}
\noindent
Proceeding similarly as in the proof of \cite[Prop. 2.6]{Zieg2} (or Proposition \ref{lambdaap}) and using Proposition \ref{strictext} one shows the following.
\begin{prop}\label{apf}
The class of $F$-fields has the amalgamation property.
\end{prop}
\noindent
We can conclude now as in Section \ref{sececg}.
\begin{theorem}
The theory $\fdcf$ is a model completion of the theory $\fdf$ (so it eliminates quantifiers).
\end{theorem}
\begin{remark}
\begin{enumerate}
\item The extra property which makes the theory $\fdf$ nicer than the theory $\gdf$ is the existence of the \emph{smallest} strict extensions (Proposition \ref{strictext}). It gives the amalgamation property for \emph{all} (i.e. not necessarily strict) $F$-fields (Proposition \ref{apf}) and the quantifier elimination for $\fdcf$.

\item In this context, quantifier elimination for $\fdcf$ implies elimination of imaginaries for $\fdcf$ exactly as in Section 4 of \cite{Zieg2}.
\end{enumerate}
\end{remark}

\begin{theorem}\label{chain}
The theories $F[m]-\dcf$ form an increasing chain and we have
$$F-\dcf = \bigcup_{m=1}^{\infty}F[m]-\dcf.$$
\end{theorem}
\begin{proof}
It follows just by inspecting the axioms of the theories in question.
\end{proof}

\subsection{Bi-interpretability with a theory of separably closed fields}
Clearly, each model of $\fdcf$ restricts to a model of $\scf_{p,e}$. In this subsection, we discuss the opposite problem: can any model of $\scf_{p,e}$ be expanded to a model of $\fdcf$? The same question can be asked for $\mathfrak{g}$ in place of $F$. Ziegler showed in \cite{Zieg2} that the answer is affirmative for $F=\widehat{\ga^e}$. The second author showed the same in \cite{K3} for $\mathfrak{g}=\ga^e[m]$. In this subsection, we generalize the above results to the case when $F$ is of the form $\widehat{G}$ and $\mathfrak{g}$ of the form $G[m]$.

We actually show more, i.e. we will see that (after adding some extra constants) the theory $G-\dcf$ is \emph{extension by definitions} (see \cite[page 59]{shologic}) of the theory $\scf_{p,e}$.

By $L$, we denote the language of $k$-algebras, and by $\scf_{p,e}$, the theory of separably closed $k$-algebras with the degree of imperfection $e$
in the language $L$ (so our notation does not reflect the dependence on $k$, we hope it will not cause any confusion). Recall that $L_e$ denotes the language of $k$-algebras with $e$-dimensional HS-derivations. We introduce two new languages (obtained after adding $e$ extra constant symbols):
$$L^{\mathbf{b}}:=L\cup\lbrace b_1,\ldots,b_e\rbrace,\qquad L_e^{\mathbf{b}}:=L_e\cup\lbrace b_1,\ldots,b_e\rbrace.$$
Let $\beta$ be a sentence in the language $L^{\mathbf{b}}$ saying that $b_1,\ldots,b_e$ form a $p$-basis.
Now we add $\beta$ to the theory $\scf_{p,e}$ and to the theory $F-\dcf$ to obtain the theory
$$\scf_{p,e}^{\mathbf{b}}:=\scf_{p,e}\;\cup\;\lbrace\beta\rbrace$$
in the language $L^{\mathbf{b}}$, and the theory
$$F-\dcf^{\mathbf{b}}:=F-\dcf\;\cup\;\lbrace\beta\rbrace$$
in the language $L_e^{\mathbf{b}}$.
\begin{lemma}\label{finsep}
There are $x_1,\ldots,x_e\in k(G)$ algebraically independent over $k$ such that the field extension $k(x_1,\ldots,x_e)\subseteq k(G)$ is finite and separable. %{\bf we need these $b_1,\ldots,b_e$ to be an arbitrary $p$-basis!???}
\end{lemma}
\begin{proof}
Let $x_1,\ldots,x_e\in \mathcal{O}_G$ be a sequence of local parameters (see Example \ref{cangex}). By Example \ref{cangex}, \cite[30.6(ii)]{mat} and Proposition on page 276 of \cite{mat0}, $\{x_1,\ldots,x_e\}$ is a $p$-basis of $\mathcal{O}_G$, so it is also a $p$-basis of $k(G)$. Hence the field extension $k(x_1,\ldots,x_e)\subseteq k(G)$ is finite and separable.
\end{proof}
Let $x_1,\ldots,x_e$ be as in Lemma \ref{finsep}. By Abel's theorem, there is $y\in k(G)$ such that $k(G)=k(x_1,\ldots,x_e,y)$. Let $H(X_1,\ldots,X_e,Y)\in k[X_1,\ldots,X_e,Y]$ be such that $H(x_1,\ldots,x_e,y)=0$ and the polynomial $H(x_1,\ldots,x_e,Y)$ is irreducible.

We need to add one more constant to the languages we consider to obtain the following languages:
$$L^{{\mathbf{b}},c}:=L\cup\lbrace b_1,\ldots,b_e,c\rbrace,\qquad L_e^{{\mathbf{b}},c}:=L_e\cup\lbrace b_1,\ldots,b_e,c\rbrace.$$
The meaning of this extra constant $c$ is that it generates the field $k(G)$ over $k$ extended by the chosen $p$-basis.
Formally, we specify one more axiom $\gamma$ in the language $L^{{\mathbf{b}},c}$  
$$\gamma:=\ \ \beta \ \wedge  \ (H(b_1,\ldots,b_e,c)=0).$$
We define the following theories.
$$\scf_{p,e}^{{\mathbf{b}},c}:=\scf_{p,e}\;\cup\;\lbrace \gamma\rbrace,\ \ F-\dcf^{{\mathbf{b}},c}:=F-\dcf\;\cup\;\lbrace \gamma\rbrace .$$

\begin{theorem}\label{biin}
The theory $G-\dcf^{{\mathbf{b}},c}$ is an extension by definitions of the theory $\scf^{{\mathbf{b}},c}_{p,e}$.
\end{theorem}
\begin{proof}
Take $x_1,\ldots,x_e\in k(G)$ as in Lemma \ref{finsep}, denote $(x_1,\ldots,x_e)$ by $\mathbf{x}$ and let $y\in k(G)$ be separable algebraic over $k(\mathbf{x})$ such that $k(\mathbf{x},y)=k(G)$ (such $y$ exists by Lemma \ref{finsep} and Abel's theorem). Then for any $\mathbf{i},\mathbf{j}\in \Nn^e$, there are polynomials $F_{\mathbf{j},\mathbf{i}},H_{\mathbf{j},\mathbf{i}}\in k[\mathbf{X},Y]$ such that for the canonical $F$-derivation
$\mathbb{D}^{F}=(D_{\mathbf{j}}^F)_{\mathbf{j}\in\mathbb{N}^e}$ on $k(G)$, we have
$$D_{\mathbf{j}}^F(\mathbf{x}^{\mathbf{i}})=\frac{F_{\mathbf{j},\mathbf{i}}(\mathbf{x},y)}{H_{\mathbf{j},\mathbf{i}}(\mathbf{x},y)}\in k(G).$$
We describe now an extension by definitions of the theory $\scf_{p,e}^{{\mathbf{b}},c}$ coinciding with the theory $G-\dcf^{{\mathbf{b}},c}$. To ease the notation, we will also denote by $\mathbf{b}$ the tuple (of constant symbols or elements of a model) $(b_1,\ldots,b_e)$.
Take $\mathbf{j}=(j_1,\ldots,j_e)\in\mathbb{N}^e$, for every $n\geqslant \max\lbrace j_1,\ldots,j_e\rbrace$
we add to the theory $\scf_{p,e}^{{\mathbf{b}},c}$ the following defining axiom
\begin{equation}\label{scfD}\tag{$\clubsuit$}
D_{\mathbf{j}}(x)=y \leftrightarrow \left(\exists^{\mathbf{i}\in[p^n]^e}\; \alpha_{\mathbf{i}}\right)
\left(x=\sum\limits_{\mathbf{i}\in[p^n]^e}\alpha_{\mathbf{i}}^{p^n}\mathbf{b}^{\mathbf{i}}
\;\wedge\; y=\sum\limits_{\mathbf{i}\in[p^n]^e}\alpha_{\mathbf{i}}^{p^n}\frac{F_{\mathbf{j},\mathbf{i}}(\mathbf{b},c)}{H_{\mathbf{j},\mathbf{i}}(\mathbf{b},c)} \right).
\end{equation}
%where $[p^n]:=\lbrace 0,\ldots,p^n-1\rbrace$.
Take any $(K,\mathbf{b},c)\models \scf_{p,e}^{{\mathbf{b}},c}$.  It is not hard to verify that the defining axioms (\ref{scfD}) determine well-defined functions on $K$. The axioms $\beta,\gamma$ guarantee that
 $$k(\mathbf{b},c)\cong_k k(G)$$ 
 and that $\mathbb{D}:=(D_{\mathbf{j}})_{\mathbf{j}\in \Nn^e}$ restricted to $k(\mathbf{b},c)$ corresponds to the canonical $G$-derivation on $k(G)$.
 Because $\mathbf{b}$ is a $p$-basis of $K$, the extension
 $k(\mathbf{b})\subseteq K$ is \'{e}tale
 (see \cite[Theorem 26.8]{mat}). Hence the extension $k(\mathbf{b},c)\subseteq K$ is also \'{e}tale. Due to Proposition \ref{etalehs}, we can extend the canonical $G$-derivation from $k(G)$ to the field $K$. Let $\mathbb{D}'=\left(D'_{\mathbf{j}}\right)_{\mathbf{j}\in\mathbb{N}^e}$
 denote the $G$-derivation on $K$ extending the canonical $G$-derivation on $k(G)$.

 Note that $D'_{\mathbf{j}}(x)=D_{\mathbf{j}}(x)$ for every $x\in K$, hence $\mathbb{D}'=\mathbb{D}$.
 Moreover the canonical $G$-derivation on $k(G)$ is strict, so also $\mathbb{D}'|_{k(\mathbf{b},c)}$ is strict.
 Lemma \ref{constantdisjoint} implies that $\mathbb{D}'=\mathbb{D}$ is strict, so $(K,\mathbb{D},\mathbf{b},c)\models G-\dcf^{{\mathbf{b}},c}$.
%Let $B$ be a $p$-basis of $K$. Then $|B|=e$ and $B$ is algebraically independent over $k$. Let $x_1,\ldots,x_e\in \mathcal{O}_G$ be a sequence of local %parameters (see Example \ref{cangex}). By Example \ref{cangex}, \cite[30.6(ii)]{mat} and Proposition on page 276 of \cite{mat0}, $\{x_1,\ldots,x_e\}$ is a %$p$-basis of $\mathcal{O}_G$, so it is also a $p$-basis of $k(G)$. Hence the extension $k(x_1,\ldots,x_e)\subseteq k(G)$ is separable algebraic. We embed $k(G)$ %into $K$ over $k$ by mapping $\{x_1,\ldots,x_e\}$ onto $B$. Since $B$ is a $p$-basis of both $k(G)$ and $K$, the extension $k(G)\subseteq K$ is \'{e}tale.
%\\
%By Example \ref{cangex}, there are canonical $G$-derivations on $k(G)$ and $k((x_1,\ldots,x_e))$ such that $k(G)\subseteq k((x_1,\ldots,x_e))$ is a %$G$-extension. Since $x_1,\ldots,x_e$ is a $p$-basis of both $k(G)$ and $k((x_1,\ldots,x_e))$, the extension $k(G)\subseteq k((x_1,\ldots,x_e))$ is \'{e}tale. By %Proposition \ref{canstrict}, $k((x_1,\ldots,x_e))$ with the canonical $G$-derivation is strict. By Lemma \ref{constantdisjoint}(3), $k(G)$ with the canonical %$G$-derivation is also strict. Since the extension $k(G)\subseteq K$ is \'{e}tale, by Proposition \ref{etalegext} and Lemma \ref{constantdisjoint}(3) again, %there is a $G$-derivation $\mathbb{D}$ on $K$ such that $(K,\mathbb{D})$ is strict and $(K,\mathbb{D})\models G-\dcf$.
\end{proof}
\noindent
%Similarly as in \cite{Zieg2} one shows the following. {\bf CHECK!!}
%\begin{theorem}
%The theory $G-\dcf$ (and $G[m]-\dcf$) is bi-interpretable with $\scf_{p,e}$ after naming $e$ constants.
%\end{theorem}
%\noindent
\begin{remark}
In many cases, the theory $G-\dcf^{\mathbf{b}}$ is an extension by definitions of the theory $\scf^{{\mathbf{b}}}_{p,e}$, so in such cases the constant symbol $c$ is not necessary. For example, it is the case for $G$ being the group of Witt vectors, see \cite{Hoff2}.
\end{remark}
It is unclear to us how to proceed in the case of an arbitrary formal group $F$. The crucial question is whether the canonical derivation on $k((\mathbf{X}))$ can be restricted to $k(\mathbf{X})$ or to $k(\mathbf{X})^{\sep}$. This question has been investigated in \cite{HK2} (it is related to Matsumura's \emph{integrability question}). To prove a partial result (Theorem \ref{onedimext}), we need the following lemma.
\begin{lemma}\label{isodef}
Assume that $F_1$ and $F_2$ are one-dimensional formal group laws over $k$ such that $F_1\cong F_2$. Then the theory $F_1-\dcf$ is an extension by definitions of the theory $F_2-\dcf$.
\end{lemma}
\begin{proof}
Let $f\in Xk\llbracket X\rrbracket$ be an isomorphism between $F_1$ and $F_2$. For any $F_1$-derivation $\mathbb{D}:K\to K\llbracket X\rrbracket$, the following composition
\begin{equation*}
\xymatrix{K \ar[r]^{\mathbb{D}}  & K\llbracket X\rrbracket \ar[r]^{\ev_f}  & K\llbracket X\rrbracket}
\end{equation*}
is an $F_2$-derivation. It is easy to see that for each positive integer $n$, there are $c_{n,1},\ldots,c_{n,n}\in k$ such that if $\mathbb{D}=(D_i)_{i\in \Nn}$ then the corresponding $F_2$-derivation $\ev_f\circ \mathbb{D}$ is of the form
$$(\id,c_{1,1}D_1,c_{2,1}D_1+c_{2,2}D_2,\ldots),$$
hence the result follows.
\end{proof}
Using \cite{LubinAns} and \cite{Bakur} (which we needed to prove the crucial \cite[Thm. 4.3]{HK2}), we can now show the following.
\begin{theorem}\label{onedimext}
Suppose that $F$ is a one-dimensional formal group law over an algebraically closed field $k$ (in particular $e=1$). Then the theory $F-\dcf^{\mathbf{b}}$ is an extension by definitions of the theory $\scf^{\mathbf{b}}_{p,1}$.
\end{theorem}
\begin{proof}
%Let $(K,\mathbf{b})\models \scf^B_{p,1}$. Since $k$ is perfect and $K$ separably closed, the algebraic closure of $k$ is contained in $K$. Hence, without loss of %generality, we can assume that $k$ is algebraically closed.
By \cite[Thm. 4.3]{HK2}, there is a formal group law $\widetilde{F}$ over $k$ such that $F\cong \widetilde{F}$ and the canonical $\widetilde{F}$-derivation restricts to $k[X]$. By Lemma \ref{isodef}, we can assume that $F=\widetilde{F}$. We can repeat now the proof of Theorem \ref{biin} (and we do not need to worry about the extra constant $c$).
\end{proof}
\begin{remark}
Passing from the base field $k$ to its algebraic closure expands the language with extra constants. But any existentially closed $F$-field is separably closed, so its field of absolute constants contains $k^{\alg}$. Hence such an expansion is not important for the theories we consider, since the models of $F-\dcf$ in the language with constants for elements of $k$ are the same as models of $F-\dcf$ in the language with constants for elements of $k^{\alg}$. Therefore, the assumption in Theorem \ref{onedimext} that $k$ is algebraically closed is harmless.
\end{remark}

%can be expanded to a model of $\fdcf$ for a non-algebraic $F$. We show in \cite{HK2} that if there is a non-trivial one-dimensional $F$-derivation on the %projective line, then $F\cong \widehat{\ga}$ or $F\cong \widehat{\gm}$. This result suggests that for a non-algebraic $F$, the theory $\fdcf$ may be describing a %proper subclass of models of $\scf_{p,e}$. It would be interesting to characterize such a class algebraically.
%It would be interesting to find a formal group $F$, for which $F-\dcf$ is \emph{not} bi-interpretable with the theory of separably closed fields of degree of %imperfection  $e$. One can also consider $F$-fields which are localizations of $k$-algebras of finite type. For example, there is a non-algebraic one-dimensional %formal group $F$ embedding into the formalization of an Abelian variety $A$ (see \cite{Manin}), so the field of rational functions $k(A)$ has a natural structure %of an $F$-field.
%{\bf IF THERE IS AN EXPANSION OF $k(\mathbf{X})^{\sep}$ TO A MODEL OF $\fdcf$, THEN $F=\widehat{G}$ (FOR SOME ALGEBRAIC GROUP $G$)???}

\section{Geometric axioms}\label{secaxioms}
\noindent
In this section, we give geometric axioms for the theories $\mathfrak{g}-\dcf$ and $F-\dcf$.  The presentation follows the one in \cite{K3}, however (unlike in \cite{K3}) we notice here that the existence of canonical $p$-bases (see Section \ref{secgtup}) is \emph{not} necessary for the geometric axioms. We do \emph{not} assume in this section that $\mathfrak{g}$ is of the form $W[m]$ for a formal group law $W$.

\subsection{Prolongation and comultiplication}\label{secproco}
The notions introduced in this subsection are special cases of the notions considered in \cite{MS}. These notions originate from Buium \cite{Bui2} and also appeared (among others) in \cite{PiZie} and \cite{KP5}.
\\
\\
We fix a field $K$ with a $\mathfrak{g}$-derivation $\mathbb{D}$. Our first definitions do not use the $\mathfrak{g}$-iterativity condition.
\begin{itemize}
\item Let $\mathcal{D}$ be the functor from the category of $K$-algebras to the category of $K[\mathbf{v}_m]$-algebras defined in the following way:
$$\mathcal{D}(R)=R\otimes_{K,\mathbb{D}}K[\mathbf{v}_m].$$
Since $\mathcal{D}$ commutes with localizations, it also defines a functor from $K$-schemes to $K[\mathbf{v}_m]$-schemes.

\item The functor $\mathcal{D}$ considered as a functor from $K$-algebras to $K$-algebras has a left-adjoint functor $\nabla$ which extends to $K$-schemes. A crucial natural bijection is the following one:
    $$(\nabla V)(R)\longleftrightarrow V(\mathcal{D}(R)).$$

\item For any $K$-scheme $V$ we have a (non-algebraic!) map
$$\mathbb{D}_{V}:V(K)\to V(\mathcal{D}(K))=\nabla V(K)$$
induced by the $K$-algebra homomorphism $\mathbb{D}:K\to K[\mathbf{v}_m]=\mathcal{D}(K)$.
\end{itemize}
\begin{remark}
Our notation here differs from the notation used in \cite{Bui2} and \cite{MS}, where the left-adjoint functor considered above is denoted $\tau$, and the  notation $\nabla$ is used for $\mathbb{D}_{V}$.
\end{remark}

\noindent
The second set of definitions uses the $\mathfrak{g}$-iterativity condition.
\begin{itemize}

\item We have a natural transformation (of functors on the category of $K$-algebras) given by the Hopf algebra comultiplication (coming from $\mathfrak{g}$)
$$\mu:\mathcal{D}\to \mathcal{D}\circ \mathcal{D}.$$

\item We define a natural transformation of functors on the category of $K$-schemes
$$c:\nabla\to \nabla\circ \nabla$$
using the commutative diagram below
\begin{equation*}
\xymatrixcolsep{4pc}
\xymatrix{V(\mathcal{D}(R))  \ar[r]^{V(\mu)} &  V(\mathcal{D}(\mathcal{D}(R)))\\
  \nabla V(R) \ar[u]^{\cong} \ar[r]^{c_V} &  \nabla(\nabla V)(R) \ar[u]^{\cong}. }
\end{equation*}

\end{itemize}
\noindent
Below we give explicit descriptions of the maps $\mathbb{D}_V$ and $c_V$.
\\
For any positive integer $n$, our HS-derivation $\mathbb{D}$ naturally extends to the following HS-derivation
$$\left(D_{\mathbf{j}}:K[X_1,\ldots,X_n]\to K[X_1^{(\mathbf{i})},\ldots,X_n^{(\mathbf{i})}\ |\ \mathbf{i}\in[p^m]^e\ ]\right)_{\mathbf{j}\in[p^m]^e},\ \ \mathbb{D}_{\mathbf{j}}(X)=X^{(\mathbf{j})},$$
where $X_k^{(\mathbf{0})}=X_k$ and for $\mathbf{j}\neq \mathbf{0}$, $X_k^{(\mathbf{j})}$ is a new variable. We will use the following notation $$K\{X_1,\ldots,X_n\}:=K[X_1^{(\mathbf{i})},\ldots,X_n^{(\mathbf{i})}\ |\ \mathbf{i}\in[p^m]^e\ ].$$
If $R=K[X_1,\ldots,X_n]/I$, then $\nabla(R)=K\{X_1,\ldots,X_n\}/(\mathbb{D}(I))$. Hence, for an affine variety $V=\spec(R)$, the variety $\nabla(V)$ is defined by the ideal $(\mathbb{D}(I))$ and $\mathbb{D}_V$ is given in coordinates as the $n$-th Cartesian product of $\mathbb{D}$ (considered as a map from $K$ to $K^{p^{me}}$).
\\
Let $c_n$ denote $c_{\Aa^{np^{me}}}$. For every $(b_{\mathbf{i},1},\ldots,b_{\mathbf{i},n})_{\mathbf{i}\in [p^m]^e}\in K^{np^{me}}$ we have
$$c_n\left((b_{\mathbf{i},1},\ldots,b_{\mathbf{i},n})_{\mathbf{i}\in [p^m]^e}\right)=
\left(\sum\limits_{\mathbf{k}\in[p^m]^e}c_{\mathbf{i},\mathbf{j}}^{\mathbf{k}}b_{\mathbf{k},1},\ldots,
\sum\limits_{\mathbf{k}\in[p^m]^e}c_{\mathbf{i},\mathbf{j}}^{\mathbf{k}}b_{\mathbf{k},n}\right)_{\mathbf{i},\mathbf{j}\in [p^m]^e},$$
where $c_{\mathbf{i},\mathbf{j}}^{\mathbf{k}}$ are the ``structural constants'' from Remark \ref{gderint2}.
We will need the following.
\begin{lemma}\label{nablaL}
Let $V$ be a $K$-scheme and suppose that $(K,\mathbb{D})\subseteq(L,\mathbb{D}')$ is an extension of $m$-truncated $e$-dimensional HS-fields. Then for any $a\in V(L)$, we have
 $\mathbb{D}'_V(a)\in\nabla V(L)$.
\end{lemma}
\begin{proof}
It is enough to notice that if $\mathbb{D}'$ extends $\mathbb{D}$, then the $k$-algebra map $\mathbb{D}':L\to \mathcal{D}(L)$ is $K$-linear.
\end{proof}
\noindent
The following lemma corresponds to \cite[Lemma 1.1(ii)]{K3} and is a direct consequence of the $\mathfrak{g}$-iterativity condition.
\begin{lemma}\label{gitnabla}
For any $K$-scheme $V$ we have the following:
$$\mathbb{D}_{\nabla(V)}\circ \mathbb{D}_V=c_V\circ \mathbb{D}_V.$$
\end{lemma}
\noindent
Let us fix a $|K|^+$-saturated algebraically closed field $\Omega$ containing $K$. We want to describe possible extensions of $\mathbb{D}$ to subfields of $\Omega$ in terms of the functor $\nabla$. Let $b_0\in\Omega^n$, $b=(b_0,\ldots,b_{p^{me}-1})\in\Omega^{np^{me}}$ and we set
$$V=\locus_K(b_0),\ \ \ W=\locus_K(b).$$
As in \cite[Lemma 3.3]{K3}, we can show the following.
\begin{lemma}\label{nablaext}
If $b\in \nabla V(\Omega)$ and $K(b_0)=K(b)$, then there is an $m$-truncated $e$-dimensional HS-derivation $\mathbb{D}'$ on $K(b)$ extending $\mathbb{D}$ such that $\mathbb{D}'_V(b_0)=b$.
\end{lemma}
\noindent
Finally, we obtain the following lemma by using Lemmas \ref{nablaL}, \ref{gitnabla}, \ref{nablaext} (as in \cite{K3}).
\begin{lemma}\label{cVW}
The following are equivalent.
\begin{enumerate}
\item  There is a $\mathfrak{g}$-field extension $(K,\mathbb{D})\subseteq(K(b),\mathbb{D}')$ such that $\mathbb{D}'_V(b_0)=b$.

\item  There is a $\mathfrak{g}$-field extension $(K,\mathbb{D})\subseteq(L,\mathbb{D}')$ such that $\mathbb{D}'_V(b_0)=b$.

\item $c_{n}(W)\subseteq \nabla(W)$.
\end{enumerate}
\end{lemma}

\subsection{Geometric axioms}\label{secgag}
 In this subsection, we give geometric axioms for the theories $\gdcf$ and $\fdcf$. In the case of $F=\widehat{\ga^e}$, we will recover (and actually we will also correct, thanks to referee's comment, by adding the assumption that $K$ is separably closed) the geometric axioms for $\sch_{p,e}$ from \cite[Theorem 4.3]{K3}. First we deal with the truncated case.
\bigskip
\newline
{\bf Geometric axioms for $\gdcf$}
\begin{enumerate}
\item The field $K$ is separably closed.

\item For each positive integer $n$, suppose that $V\subseteq {\mathbb A}^n$ and $W\subseteq \nabla(V)$ are
absolutely irreducible $K$-varieties, and $Z$ is a proper $K$-subvariety
of $W$. If $W$ projects generically onto $V$, and $c_{V}(W)\subseteq
\nabla (W)$, then there is $a\in V(K)$ such that $\mathbb{D}_{V}(a)\in W(K)\setminus Z(K)$.
\end{enumerate}
\noindent
These axioms are first-order, since for a separably closed field $K$, any $K$-irreducible variety is absolutely irreducible.
\begin{theorem}\label{ggeo}
The $\mathfrak{g}$-field $(K,\mathbb{D})$ is an existentially
closed $\mathfrak{g}$-field if and only if $(K,\mathbb{D})$ satisfies the geometric axioms above.
\end{theorem}
\begin{proof}
 Assume that $(K,\mathbb{D})$ is an existentially closed $\mathfrak{g}$-field. Then by Proposition \ref{etalegext}, $K$ is separably closed (since a separable algebraic field extension is \'{e}tale). Take any $K$-irreducible $K$-varieties $V\subseteq \mathbb{A}^n$, $W\subseteq \nabla(V)$ and
 $Z\subsetneq W$. If $W$ projects generically onto $V$ then there is $b_0\in\Omega^n$ and
 $b=(b_0,\ldots,b_{p^{me}-1})\in\Omega^{np^{me}}$ such that $V=\locus_K(b_0)$ and $W=\locus_K(b)$. Since $c_V(W)\subseteq \nabla(W)$,
 Lemma \ref{cVW} implies that there exists a $\mathfrak{g}$-field extension $(K,\mathbb{D})\subseteq (K(b),\mathbb{D}')$
 such that $\mathbb{D}'_V(b_0)=b$. We have
 $$\mathbb{D}'_V(b_0)=b\in W(K(b))\setminus Z(K(b))$$
 and $b_0\in V(K(b))$. Since
 $(K,\mathbb{D})$ is existentially closed, there is $b_0'\in V(K)$ such that
 $\mathbb{D}_V(b'_0)\in W(K)\setminus Z(K)$.
 \\
 \\
Assume now that $(K,\mathbb{D})$ is a model of the ``geometric axioms for $\gdcf$'' and let $(K,\mathbb{D})\subseteq (L,\mathbb{D}')$ be an extension of $\mathfrak{g}$-fields. Take any quantifier-free formula $\varphi(x_0,\ldots,x_{p^{me}-1})$ over $K$ (in the language of fields), where each $x_i$ is an $n$-tuple of variables for an arbitrary (but fixed) positive integer $n$. Suppose that $(L,\mathbb{D}')\models (\exists x_0)\varphi(\mathbb{D}'(x_0))$. Take $b_0\in L^n$ such that $(L,\mathbb{D}')\models\varphi(\mathbb{D}'(b_0))$
 and set
 $$V=\locus_K(b_0),\ \ b=\mathbb{D}_V'(b_0),\ \ W=\locus_K(b),\ \ Z_0=\lbrace d\in W\;\;|\;\; \neg\varphi(d)\rbrace.$$
Let $Z$ denote the intersection of all $K$-subvarieties of $W$ containing $Z_0$. Clearly $b\in W(L)\setminus Z(L)$ and
the second condition of Lemma \ref{cVW} is satisfied, so $c_{V}(W)\subseteq \nabla(W)$. Hence all the assumptions of the ``geometric axioms for $\gdcf$'' hold.
 Therefore there is $b_0'\in V(K)$ such that $\mathbb{D}(b_0')\in W(K)\setminus Z(K)$ and $(K,\mathbb{D})\models (\exists x_0)\varphi(\mathbb{D}(x_0))$.
\end{proof}
\begin{remark}\label{gremark}
Note that the results of Section \ref{secmc} are not used in the proof of Theorem \ref{ggeo} which suggests a possibility of generalizations, e.g. to the context of $\mathcal{D}$-fields from \cite{MS}.
\end{remark}
\noindent
We turn now to the case of $F$-derivations and assume that $(K,\mathbb{D})$ is an $F$-field. By Theorem \ref{chain}, the geometric axioms for $\fdcf$ are given as the union (over $m$) of the geometric axioms for $F[m]-\dcf$. We state these axioms explicitly below, where $\nabla_m$ denotes the functor $\nabla$ with respect to $\mathbb{D}[m]$ (similarly for $c_m$).
\\
\\
{\bf Geometric axioms for $\fdcf$}
\begin{enumerate}
\item The field $K$ is separably closed.

\item For any positive integers $n,m$, suppose that $V\subseteq {\mathbb A}^n$ and $W\subseteq \nabla_m(V)$ are
absolutely irreducible $K$-varieties, and $Z$ is a proper $K$-subvariety
of $W$. If $W$ projects generically onto $V$, and $c_{m,V}(W)\subseteq
\nabla_m(W)$, then there is $a\in V(K)$ such that $\mathbb{D}[m]_{V}(a)\in W(K)\setminus Z(K)$.
\end{enumerate}
\noindent
We get a result generalizing \cite[Theorem 4.3]{K3}.
\begin{theorem}\label{fgeo}
The $F$-field $\mathbf{K}$ is an existentially
closed $F$-field if and only if $\mathbf{K}$ is a model of the geometric axioms for $\gdcf$.
\end{theorem}
\begin{remark}\label{fremark}
Unlike in the proof of Theorem \ref{ggeo}, the results of Section \ref{secmc} \emph{are} used for the proof of Theorem \ref{fgeo}, since Theorem \ref{chain} is necessary for the geometric axiomatization and the proof of Theorem \ref{chain} requires the algebraic axiomatizations of $\gdcf$ and $\fdcf$. However, one can also prove Theorem \ref{chain} (but for a limited class of formal groups $F$ only) in another fashion as it was done in \cite{K3} for $F=\widehat{\ga^e}$. This approach will be discussed in Section \ref{secgtup}.
\end{remark}

\section{Fields with operators and canonical $G$-tuples}\label{msecop}
\noindent
As we have mentioned several times, $F$-iterative fields fit to the more general set-up of \emph{iterative $\underline{\mathcal{D}}$-fields}, see \cite{MS}. Model companions of the theories of iterative $\underline{\mathcal{D}}$-fields are analyzed in \cite{MS2}, however only the case of characteristic $0$ and only the ``trivial'' iterativity maps (see \cite[Section 6.1]{MS2}) are considered there. One could wonder whether our techniques may be generalized to include the case of iterative $\underline{\mathcal{D}}$-fields of positive characteristic. It seems likely in the case where an iterative system $(\underline{\mathcal{D}},\Delta)$ is the inverse limit of finite iterative systems $(\mathcal{D}_m,\Delta_m)_m$ resembling the ones given by Hopf algebra comultiplications $\mu_m:\mathcal{D}_m\to \mathcal{D}_m\circ \mathcal{D}_m$ (see Section \ref{secproco}). More precisely, our techniques may apply to the iterative systems $(\underline{\mathcal{D}},\Delta)$ where each morphism $\Delta_{p^m}:\mathcal{D}_{2p^m}\to \mathcal{D}_{p^m}\circ \mathcal{D}_{p^m}$ factors through the projection morphism $\mathcal{D}_{2p^m}\to \mathcal{D}_{p^m}$. It is easy to see that each direct system of finite group schemes provides such an iterative system, and, for example, \'{e}tale group schemes (corresponding to actual groups) would correspond to systems governing actions of groups on fields by field automorphisms.
\\
\\
A geometric axiomatization of the class of existentially closed $(\mathcal{D}_m,\Delta_m)$-fields should not be very difficult, see Remark \ref{gremark}. The crucial technical point allowing a geometric axiomatization of the class of existentially closed $(\underline{\mathcal{D}},\Delta)$-fields may require an appropriate generalization of Theorem \ref{chain}. The proof given in this paper seems to be too specific for a possible generalization to this more general context. However, one could have proceeded in Section \ref{secaxioms} differently, more in the fashion of \cite{K3} where Ziegler's notion of a \emph{canonical $p$-basis} is used (however we would get then Theorem \ref{fgeo} only for a limited class of formal groups $F$). We sketch this approach below.

\subsection{Canonical $G$-tuples}\label{secgtup}
We generalize the notion of a canonical $p$-basis (see \cite{Zieg2}) from the case of the formalization of a vector group to the case of the formalization of an arbitrary algebraic group.
\begin{definition}\label{defpbasis}
Let $\mathbb{D}$ be a $G[m]$-derivation on $K$.
%, $k\subseteq F\subseteq K$, $F$ perfect.
A subset $B\subseteq K$ is called a \emph{canonical $G$-tuple}, if $|B|=e$ and there is a $G[m]$-embedding $(k(G),\mathbb{D}[m]^G)\to \mathbf{K}$ such that $B$ is the image of the set of canonical parameters of $G$, where $\mathbb{D}[m]^G$ is the canonical $G[m]$-derivation from Example \ref{cangex}.
\end{definition}
\begin{remark}\label{geotuple}
If $\mathbf{L}$ is strict and of imperfection degree $e$, then any canonical $G$-tuple in $L$ is a $p$-basis. For $G=\ga^e$, we recover the notion of a \emph{canonical $p$-basis} from \cite{Zieg2}.
\end{remark}
\noindent
We define below a general property of algebraic groups.
\begin{definition}\label{existence}
We say that \emph{canonical $G$-tuples exist} if for any $m$ and any separably closed $G[m]$-field $\mathbf{L}$, whenever $[L:C_L]=p^e$, there is a canonical $G$-tuple in $L$.
\end{definition}
\begin{remark}\label{greater}
\begin{enumerate}
\item Definition \ref{defpbasis} can be phrased in a (much) greater generality using group scheme actions.
Let $\mathfrak{G}_0$ be a group subscheme of a group scheme $\mathfrak{G}$. Assume that $\mathfrak{G}_0$ acts (as a group scheme) on a scheme $V$. We say that this action has a \emph{canonical $\mathfrak{G}$-basis}, if there is an $\mathfrak{G}_0$-invariant morphism $V\to \mathfrak{G}$ such that the induced map
$$V\to \mathfrak{G}\times_{\mathfrak{G}/\mathfrak{G}_0}V/\mathfrak{G}_0$$
is an isomorphism.

\item The existence of canonical $\mathbb{G}_a^e$-tuples is shown
in \cite{Zieg2} and the existence of
canonical $\mathbb{G}_m$-tuples is shown in \cite{HK1}. Combining these
results, one can show the existence of canonical $G$-tuples for $G$ of
the form $\mathbb{G}_a^e\times\mathbb{G}_m^f$.

\item We do not attack here the problem of the \emph{existence} of canonical $G$-tuples for a given algebraic group $G$. This will be done in \cite{Hoff1} (as well as possible applications to the notion of ``$G$-thinness'').

\item The existence of canonical $G$-tuples implies (\emph{strong}) \emph{integrability} of $G$-derivations, as in (morally) \cite{Mats1} or as in \cite{HK1} (see also \cite{Tyc})
\end{enumerate}
\end{remark}
\noindent
The existence of canonical $G$-tuples gives rather directly (see \cite[Thm. 2.3]{K3}) another proof of Theorem \ref{chain} which is the only ingredient needed in Section \ref{secaxioms} requiring specific differential-algebraic arguments. The interpretation of the notion of the existence of canonical $G$-tuples from Remark \ref{greater}(1) looks promising for possible generalizations beyond the context of HS-derivations.
\bibliographystyle{plain}
\bibliography{harvard}

\end{document}